\documentclass[11pt,reqno]{amsart}
\usepackage[dvipsnames,svgnames,x11names]{xcolor}

\usepackage{mathrsfs}

\usepackage{fullpage}
\usepackage{mathtools}
\makeatletter
\renewcommand\subsubsection{\@secnumfont}{\bfseries}%
\renewcommand\subsubsection{\@startsection{subsubsection}{3}
  \z@{.5\linespacing\@plus.7\linespacing}{-.5em}%
  {\normalfont\bfseries}}
  \makeatother

\usepackage{subfig}

\usepackage{amsfonts}
\usepackage{amsthm,amsmath}
\usepackage{amsbsy}
\usepackage{bm}
\usepackage{amssymb} 
\usepackage{wrapfig}

\usepackage{pgf,tikz}
\usepackage{graphicx}
\usetikzlibrary{trees,shadows}
\usetikzlibrary{arrows}
\usepackage{enumerate,enumitem}
\usepackage{bbm}
\usepackage{mathtools}

\usepackage[hyperindex,breaklinks]{hyperref}
\hypersetup{ colorlinks=true, linkcolor=blue, citecolor=blue,
  filecolor=blue, urlcolor=blue}
  \numberwithin{equation}{section} 
 \makeatletter
\def\paragraph{\@startsection{paragraph}{4}%
  \z@\z@{-\fontdimen2\font}%
  {\bfseries\itshape}}
  \makeatother

\DeclareMathOperator{\one}{\mathbbm{1}} 

\usepackage[comma, sort&compress]{natbib}

\newtheoremstyle{newplain}%
  {}{}
  {\itshape}{}
  {\bfseries}{.}
  { }{\thmname{#1}\thmnumber{ #2}\thmnote{ (#3)}}

\theoremstyle{newplain}
    \newtheorem{theorem}{Theorem}[section]
    \newtheorem{lemma}[theorem]{Lemma}
    \newtheorem{proposition}[theorem]{Proposition}

\theoremstyle{definition} 
    \newtheorem{definition}[theorem]{Definition}
    \newtheorem{fact}[theorem]{Fact}

\DeclareMathOperator{\gauss}{\mathrm{N}}
\DeclareMathOperator{\Ber}{\mathrm{Ber}}

\DeclareMathOperator{\Rr}{\mathbb{R}}

\DeclareMathOperator{\N}{\mathbb{N}}

\DeclareMathOperator{\ESD}{ESD}

\DeclareMathOperator{\Tr}{Tr}

\DeclareMathOperator{\tr}{tr}
\DeclareMathOperator{\diag}{diag}

\DeclareMathOperator{\Id}{I}
\DeclareMathOperator{\dd}{d}
\DeclareMathOperator{\Var}{Var}
\DeclareMathOperator{\St}{\mathrm{S}}
\newcommand{\G}{\mathbb{G}}

\newcommand{\E}{\mbox{${\mathbb E}$}}

\newcommand{\A}{\mathbf{A}}

\newcommand{\bigO}{\mbox{${\mathrm O}$}}

\newcommand{\vep}{\varepsilon}
\newcommand{\xvec}{\mathbf{x}}
\newcommand{\Xvec}{\mathbf{X}}
\newcommand{\bfd}{\mathbf{d}}
\newcommand{\bfD}{\mathbf{\Delta}}

\newcommand {\prob}{\mathbb{P}}

\newcommand{\e}{\mathrm{e}}
\newcommand{\M}{\mathbf{M}}
\newcommand{\C}{\mathbb{C}}

\newcommand{\bA}{\mathbf{A}}

\newcommand{\Ver}{\mathbf{V}}

\global\long\def\ep{\mathbf{E}} 
\global\long\def\pr{\mathbf{P}} 

\begin{document}

\title[KBRG]{Spectral properties of the Laplacian of Scale-Free Percolation models}


\author[R.~S.~Hazra]{Rajat Subhra Hazra}
 \address{University of Leiden, Mathematical Institute, Einsteinweg 55,
2333 CC Leiden, The Netherlands.}
\email{r.s.hazra@math.leidenuniv.nl}
\author[N. Malhotra]{Nandan Malhotra}
\address{University of Leiden, Mathematical Institute, Einsteinweg 55,
2333 CC Leiden, The Netherlands.}
\email{n.malhotra@math.leidenuniv.nl}

\begin{abstract}
    We consider scale-free percolation on a discrete torus $\mathbf{V}_N$ of size $N$. Conditionally on an i.~i.~d.~sequence of {Pareto} weights $(W_i)_{i\in \mathbf{V}_N}$ with tail exponent $\tau-1>0$, we connect any two points $i$ and $j$ on the torus with probability
		$$p_{ij}= \frac{W_iW_j}{\|i-j\|^{\alpha}} \wedge 1$$ for some parameter $\alpha>0$.
		We focus on the (centred) Laplacian operator of this random graph and study its empirical spectral distribution. We explicitly identify the limiting distribution when $\alpha<1$ and $\tau>3$, in terms of the spectral distribution of some non-commutative unbounded operators. 
\end{abstract}
\date{\today}
\maketitle

\section{Introduction}
In recent years, many random graph models have been proposed to model real-life networks. These models aim to capture three key properties that real-world networks exhibit: scale-free nature of the degree distribution, small-world property, and high clustering coefficients \citep{remcovol2}. It is generally difficult to find random graph models which incorporate all the three features. Classical random graph models typically fail to capture scale-freeness, small-world behavior, and high clustering simultaneously. For instance, the Erdős-Rényi model only exhibits the small-world property, while models like Chung-Lu, Norros-Reittu, and preferential attachment models are scale-free (\cite{chung2002average, barabasi1999emergence} and small-world but have clustering coefficients that vanish as the network grows. In contrast, regular lattices have high clustering but large typical distances. The Watts-Strogatz model (\cite{watts1998collective}) was an early attempt to create a network with high clustering and small-world features, but it does not produce scale-free degree distributions.

Scale-free percolation, introduced in \cite{Deijfen:Remco:Hoogh}, blends ideas from long-range percolation (see e.g.\ \cite{berger2002transience}) with inhomogeneous random graphs such as the Norros--Reittu model. In this framework, vertices are positioned on the $\mathbb{Z}^d$ lattice, and each vertex $x$ is independently assigned a random weight $W_x$. These weights follow a power-law distribution: 
\[
\mathbb{P}(W > w) = w^{-(\tau - 1)} L(w),
\]
where $\tau > 1$ and $L(w)$ is a slowly varying function at infinity.

Edges between pairs of vertices $x$ and $y$ are added independently, with a probability that increases with the product of their weights and decreases with their Euclidean distance. The edge probability is given by
\begin{equation}\label{eq:sfp}
p_{xy} = 1 - \exp\left(-\lambda \frac{W_x W_y}{\|x - y\|^\alpha}\right),
\end{equation}
where $\lambda, \alpha > 0$ are model parameters and $\|\cdot\|$ denotes the Euclidean norm. This model has been proposed as a suitable representation for certain real-world systems, such as interbank networks, where both spatial structure and heavy-tailed connectivity distributions are relevant (\cite{DHW2015}). Various properties of the model are now well known and we refer to the articles by \cite{jorritsma2024large, cipriani2024scale, CHMS2025, heydenreich2017structures, Dalmau2021} for further references. 

In recent times, there has been a lot of interest in the models which have connections probabilities similar to \eqref{eq:sfp}. Kernel-based spatial random graphs encompass a wide variety of classical random graph models where vertices are embedded in some metric space. In their simplest form (see~\cite{jorritsma2023cluster} for a more complete exposition) they can be defined as follows. Let $V$ be the vertex set of the graph and sample a collection of weights $(W_i)_{i \in V}$, which are independent and identically distributed (i.~i.~d.), serving as marks on the vertices. 
	Conditionally on the weights, two vertices $i$ and $j$ are connected by an undirected edge with probability
	\begin{equation}\label{eq:connection}
		\mathbb{P}\left(i\leftrightarrow j \mid W_i, W_j \right) = \kappa(W_i, W_j)\|i-j\|^{-\alpha} \wedge 1 \, ,
	\end{equation}
	where $\kappa$ is a symmetric kernel, $\|i-j\|$ denotes the distance between the two vertices in the underlying metric space and $\alpha>0$ is a constant parameter. In a recent article, the present authors with A. Cipriani and M. Salvi (\cite{CHMS2025}) proved the spectral properties of the adjacency matrix when $\alpha<d$ and the weights have a finite mean. In the above setting, the model was considered on a torus of side length $N$ so that the adjacency operator as a matrix was easier to describe. In this article, we aim to extend this study to the case of a Laplacian matrix. Although our approach would extend to general kernel-based models, we shall stick to the product form kernel, that is, $\kappa(x,y)=xy$, so that the ideas can be clearly presented. It is one of the few cases where the limiting distribution can be explicitly described. 

    The Laplacian of a graph with $N$ vertices is defined as $A_N- D_N$ where $A_N$ is the adjacency matrix and $D_N$ is the diagonal matrix where the $i$-th diagonal entry corresponds to the $i$-th degree. When the entries of the matrix are not restricted to $0$ or $1$, the matrix is also referred to as the \emph{Markov matrix} (\cite{BDJ2006, bordenave2014spectrum}). The graph Laplacian serves as the discrete analogue of the continuous Laplacian, essential in diffusion theory and network flow analysis. The Laplacian matrix has several key applications: The Kirchhoff Matrix-Tree Theorem relates the determinant of the Laplacian to the count of spanning trees in a graph (\cite{chung1997spectral}), the multiplicity of the zero eigenvalue indicates the number of connected components (\cite{chung1997spectral}), the second-smallest eigenvalue, known as the Fiedler value or the algebraic connectivity, measures the graph's connectivity; higher values signify stronger connectivity \cite{de2007old}. For a comprehensive overview of spectral methods in graph theory, refer to Chung's monograph \cite{chung1997spectral} and Spielman's book \cite{spielman2012spectral}. In modern machine learning, spectral techniques are pivotal in spectral clustering algorithms, where the techniques use the Laplacian eigenvalues and eigenvectors for dimensionality reduction before applying clustering algorithms like $k$-means (\cite{abbe2020graph, abbe2018community}). It is particularly effective for detecting clusters that are not linearly separable. Recent advancements integrate spectral clustering with graph neural networks to enhance graph pooling operations (\cite{bianchi2020spectral}). Spectral algorithms are also crucial for identifying communities within networks by analysing the spectral properties of the graph (\cite{chung1997spectral}).

\cite{BDJ2006} established that for large symmetric matrices with independent and identically distributed entries, the empirical spectral distribution (ESD) of the corresponding Laplacian matrix converges to the free convolution of the semicircle law and the standard Gaussian distribution. In the context of sparse Erdős–Rényi graphs, \cite{Huang:Landon} studied the local law of the ESD of the Laplacian matrix. They demonstrated that the Stieltjes transform of the ESD closely approximates that of the free convolution of the semicircle law and a standard Gaussian distribution down to the scale $N^{-1}$. Additionally, they showed that the gap statistics and averaged correlation functions align with those of the Gaussian Orthogonal Ensemble in the bulk. \cite{ding2010spectral} investigated the spectral distributions of adjacency and Laplacian matrices of random graphs, assuming that the variance of the entries of an $N\times N$ adjacency matrix depends only on $N$. They established the convergence of the ESD of these matrices under such conditions. These results of the Erd\H{o}s-R\'enyi random graphs were extended to the inhomogeneous setting by \cite{CHHS}. In a recent work, \cite{Chatterjee:Hazra} derived a combinatorial way to describe the limiting moments for a wide variety of random matrix models with a variance profile.

\paragraph{Our contribution:} The empirical spectral distribution of the (centred) Laplacian of a graph that incorporates spatial distance has not been studied before. For example, we are not aware of a result that describes the spectral properties of the Laplacian for the long-range percolation model. Our main contribution is that we establish this result for the scale-free percolation model on the torus. We restrict ourselves to the dense regime, that is, $\alpha < 1$. We show that under mild assumptions on the weights (having finite variance), we establish the existence of the limiting distribution. It turns out to be a distribution that involves the Gaussian, the semicircle, and the Pareto distribution. In a symbolic (and rather informal) way, it is given by the spectral law of $$W^{1/2}SW^{1/2}+ \sqrt{m_1} W^{1/4}G W^{1/4},$$
where $W$ is an unbounded operator with spectral law given by the Pareto distribution, $S$ is a bounded compact operator whose spectral law is the semicircle law, and $G$ is an unbounded operator whose law is given by the Gaussian distribution. Finally, $m_1$ is the first moment of $W$. The interaction between these operators comes from the fact that in the non-commutative space, $\{W, G\}$ is a commutative algebra, freely independent of $S$. Similar results have been established when the weights are bounded and degenerate, and no spatial distances are involved (\cite{Chatterjee:Hazra} and \cite{CHHS}). The present work extends the results to settings that involve random heavy-tailed weights and spatial distances. Both the conditions require many new techniques.

\paragraph{\bf Outline of the article:} In section \ref{section:set-up} we explicitly describe the setup of the model and state our main results. In Theorem \ref{theorem:Laplacian limiting measure} we show the existence of the limiting spectral distribution, and in Theorem \ref{theorem: Laplacian measure identification}, we identify the measure and state some of its properties.  In Section \ref{section:Pre-moment method} we first introduce a Gaussianised version of the matrix, and this helps us to simplify the variance profile. We then truncate the weights and decouple the diagonal, which allows us to apply the moment method. In Section \ref{sec:moment method Laplacian} we identify the limiting moments of the decoupled Laplacian and show that it does not depend on the spatial parameter $\alpha>0$, which is crucial in the identification of the limiting measure of the original Laplacian. Finally, in Section \ref{sec:final proof Laplacian} we identify the limiting measure using results from free probability. 
In Appendix \ref{appendix:prel_lemmas} we provide references to some of the results we use in our proofs, which are collections of results from other articles and are stated here for completeness.

\section{Setup and main results}\label{section:set-up}
In this section we describe the setup of the model and also state the main results. 
\subsection{Setup}
\begin{itemize}[leftmargin=*]
				\itemsep0.3em
    \item[(a)] {\bf Vertex set:} the vertex set is $\Ver_N\coloneqq \{1,\,2,\,\ldots,\,N\}$. The vertex set is equipped with torus the distance $\|i - j\| $, where
    \[
    \|i - j\| =  |i - j| \wedge (N - |i - j|).
    \]

    \item[(b)] {\bf Weights:} the weights $(W_i)_{i \in \Ver_N}$ are i.~i.~d.\ random variables sampled from a Pareto distribution $W$ (whose law we denote by $\mathbf{P}$) with parameter $\tau - 1$, where $\tau > 1$. That is,
    \begin{equation}\label{eq:paretolaw}
    \mathbf{P}(W > t) = t^{-(\tau - 1)}\one_{\{t \geq 1\}}+\one_{\{t < 1\}}.
    \end{equation}

    \item[(c)] {\bf Long-range parameter:} $\alpha > 0$ is a parameter which controls the influence of the distance between vertices on their connection probability.

    \item[(d)] {\bf Connectivity function:} conditional on the weights, each pair of distinct vertices $i$ and $j$ is connected independently with probability $P^W(i\leftrightarrow j)$ given by 
    \begin{equation}\label{connection_proba}
    P^W(i\leftrightarrow j) \coloneqq p_{ij}\coloneqq \mathbb{P}(i \leftrightarrow j \mid W_i, W_j) = \frac{W_i W_j}{\|i - j\|^\alpha} \wedge 1.
    \end{equation}
    
    We will be using the short-hand notation $p_{ij}$ for convenience. That is,
    Note that the graph does not have self-loops. 
\end{itemize}

In what follows, we denote by $\mathbb{P}$ the joint law of the weights and the edge variables. Note that $\mathbb{P}$ depends on $N$, but we will omit this dependence for simplicity. Let $\E, \mathbf{E}$, and $E^W$ denote the expectation with respect to $\prob, \mathbf{P}$, and $P^W$ respectively.

The associated graph is connected, as nearest neighbours with respect to the torus distance are always linked. Let us denote the random graph generated by our choice of edge probabilities by $\mathbb{G}_N$. Let $\mathbb{A}_{\mathbb{G}_N}$ denote the adjacency matrix (operator) associated with this random graph, defined as
\[
\mathbb{A}_{\mathbb{G}_N}(i,j) = 
\begin{cases} 
1 & \text{if } i \leftrightarrow j, \\ 
0 & \text{otherwise}.
\end{cases}
\]
Since the graph is finite and undirected, the adjacency matrix is always self-adjoint and has real eigenvalues.
Let $$\mathbb{D}_{\mathbb{G}_N}=\diag( d_1, \cdots, d_N)$$
where $d_i$ denotes the degree of the vertex $i$ and in this case given by 
\[
d_i=\sum_{j\neq i}\mathbb{A}_{\mathbb{G}_N}(i,j).
\]
We will consider the Laplacian of the matrix, which is denoted as follows:
$${\Delta}_{\mathbb{G}_N}= \mathbb{A}_{\mathbb{G}_N}- \mathbb{D}_{\mathbb{G}_N}.$$

\noindent
In general, when $0<\alpha < 1$, the eigenvalue distribution requires scaling in order to observe meaningful limiting behavior. In \cite{CHMS2025}, it was shown that an appropriate scaling of the adjacency matrix, under which the convergence of the bulk eigenvalue distribution can be studied, is given by 
\begin{equation}\label{chap5:eq:def_c}
c_N = \frac{1}{N} \sum_{i \neq j \in \Ver_N}\frac{1}{\|i-j\|^{\alpha}} \sim c_0 N^{1-\alpha},
\end{equation}
where $c_0$ is a positive constant and $0<\alpha<1$. The scaled adjacency matrix is then defined as
\begin{equation}\label{eq:scaledadjacency}
\A_N \coloneqq \frac{\mathbb{A}_{\mathbb{G}_N}}{\sqrt{c_N}}.
\end{equation}
We define the corresponding (scaled) Laplacian as
$$\mathbf{\Delta}_N= \A_N- \mathbf{D}_N,$$
where $\mathbf{D}_N$ is given by
$\mathbf{D}_N=\diag(\mathbf{d}_1, \cdots, \mathbf{d}_N)$
with \[
\mathbf{d}_i= \sum_{k\neq i} \A_N(i,k).
\]

The empirical measure that assigns a mass of $1/N$ to each eigenvalue of the $N \times N$ random matrix $\mathbf{M}_N$ is called the Empirical Spectral Distribution (ESD) of $\mathbf{M}_N$, denoted as
\[
\ESD\left(\mathbf{M}_N\right) \coloneqq \frac{1}{N} \sum_{i=1}^{N} \delta_{\lambda_i},
\]
where $\lambda_1 \leq \lambda_2 \leq \ldots \leq \lambda_{N}$ are the eigenvalues of $\mathbf{M}_N$.
We are interested in the centred Laplacian matrix for the bulk distribution. So define
\begin{equation}\label{eq:centredLap}
\mathbf{\Delta}_N^\circ=\mathbf{\Delta}_N-\E[\mathbf{\Delta}_N]
\end{equation}
where $\E[\mathbf{\Delta}_N](i,j)= \E[\mathbf{\Delta}_N(i,j)]$. If we define $\mathbf{A}_N^\circ= \A_N- \E[\A_N]$ and $\mathbf{D}_N^\circ$ is the diagonal matrix $\diag(\bfd_1^\circ, \ldots, \bfd_N^\circ)$ where $\bfd_i^{\circ}=\sum_{k\neq i} \mathbf{A}_N^\circ(i,k)$, then it is easy to see that
$$\bfD_N^\circ= \A_N^\circ- \mathbf{D}_N^\circ.$$

In this article we will be interested in understanding the behaviour of $\ESD(\bfD_N^\circ)$ as $N\to \infty$.

\subsection{Main Results}
The Lévy-Prokhorov distance $d_L: \mathcal{P}(\mathbb{R})^2 \to [0, +\infty)$ between two probability measures $\mu$ and $\nu$ on $\mathbb{R}$ is defined as
\[
d_L(\mu, \nu)
    := \inf \big\{\varepsilon > 0 \mid \mu(A) \leq \nu\left(A^\varepsilon\right) + \varepsilon \,\text{ and }\, \nu(A) \leq \mu\left(A^\varepsilon\right) + \varepsilon \quad \forall \,A \in \mathcal{B}(\mathbb{R})\big\},
\]
where $\mathcal{B}(\mathbb{R})$ denotes the Borel $\sigma$-algebra on $\mathbb{R}$, and $A^\varepsilon$ is the $\varepsilon$-neighbourhood of $A$. For a sequence of random probability measures $(\mu_N)_{N \geq 0}$, we say that
\[
\lim_{N \to \infty} \mu_N = \mu_0 \text{ in } \mathbb{P}\text{-probability}
\]
if, for every $\varepsilon > 0$,
\[
\lim_{N\to\infty}\mathbb{P}(d_L(\mu_N, \mu_0) > \varepsilon)= 0.
\]
Our first main result is existential, and is as follows. 
\begin{theorem}\label{theorem:Laplacian limiting measure}
Consider the random graph $\G_N$ on $\Ver_N$ with connection probabilities given by \eqref{connection_proba} with parameters $\tau>3$ and $0<\alpha< 1$.  Let $\ESD(\bfD_N^\circ)$ be the empirical spectral distribution of $\bfD_N^\circ$ defined in \eqref{eq:centredLap}. Then there exists a deterministic measure $\nu_{\tau}$ on $\mathbb R$ such that
    $$
    \lim_{N\to \infty}\ESD(\bfD_N^o)=\nu_{\tau}\qquad\text{ in $\,\prob$--probability}\,.
    $$
\end{theorem}

The characterisation of $\nu_\tau$ is achieved by results from the theory of free probability. For convenience, we state some technical definitions. We refer the readers to \cite[Chapter 5]{anderson2010introduction} and \cite{mingo2017free} for further details.
\begin{definition}
Let $(\mathcal{A}, \varphi )$ be a $W^*$-probability space, where $\mathcal{A}$ is a $C^*$-algebra of bounded operators on a Hilbert space (closed in the weak operator topology), and $\varphi$ is a state. A densely defined operator $T$ is said to be \textit{affiliated} with $\mathcal{A}$ if for every bounded measurable function $h$, we have $h(T) \in \mathcal{A}$. The law (or distribution) $\mathcal{L}(T)$ of such an affiliated operator $T$ is the unique probability measure on $\mathbb{R}$ satisfying
\[
\varphi(h(T)) = \int_{\mathbb{R}} h(x) \, d\mathcal{L}(T)(x).
\]

For a collection of self-adjoint operators $T_1, \ldots, T_n$, their joint distribution is described by specifying
\[
\varphi(h_1(T_{i_1}), \ldots, h_k(T_{i_k})),
\]
for all $k \geq 1$, all index sequences $i_1, \ldots, i_k \in \{1, \ldots, n\}$, and all bounded measurable functions $h_1, \ldots, h_k : \mathbb{R} \to \mathbb{R}$.
\end{definition}

\begin{definition}
Let $(\mathcal{A}, \varphi)$ be a $W^*$-probability space, and suppose $a_1, a_2 \in \mathcal{A}$. Then $a_1$ and $a_2$ are said to be \textit{freely independent} if
\[
\varphi(p_1(a_{i_1}), \ldots, p_n(a_{i_n})) = 0,
\]
for every $n \geq 1$, every sequence $i_1, \ldots, i_n \in \{1,2\}$ with $i_j \neq i_{j+1}$ for all $j = 1, \ldots, n-1$, and all polynomials $p_1, \ldots, p_n$ in one variable satisfying
\[
\varphi(p_j(a_{i_j})) = 0, \quad \text{for all } j = 1, \ldots, n.
\]
\end{definition}

\begin{definition}
Let $a_1, \ldots, a_k$ and $b_1, \ldots, b_m$ be operators affiliated with $\mathcal{A}$. The families $(a_1, \ldots, a_k)$ and $(b_1, \ldots, b_m)$ are freely independent if and only if
\[
p(h_1(a_1), \ldots, h_k(a_k)) \quad \text{and} \quad q(g_1(b_1), \ldots, g_m(b_m))
\]
are freely independent for all bounded measurable functions $h_1, \ldots, h_k$ and $g_1, \ldots, g_m$, and for all polynomials $p$ and $q$ in $k$ and $m$ non-commutative variables, respectively.
\end{definition}

We are now ready to state our second main result. 
\begin{theorem}\label{theorem: Laplacian measure identification} Under the assumptions of Theorem \ref{theorem:Laplacian limiting measure}, the limiting measure $\nu_{\tau}$ can be identified as

$$
\nu_{\tau}=\mathcal{L}\left( T_W^{1/2} T_s T_W^{1/2}+\sqrt{\E[W]} T_W^{1/4} T_g T_W^{1/4}\right).
$$

Here, $T_g$ and $T_W$ are commuting self-adjoint operators affiliated with a $W^{\star}$-probability space $(\mathcal{A}, \varphi)$ such that, for bounded measurable functions $h_1, h_2$ from $\mathbb{R}$ to itself,
$$
\varphi\left(h_1\left(T_g\right) h_2\left(T_W\right)\right)=\left(\int_{-\infty}^{\infty} h_1(x) \phi(x) d x\right)\left(\int_1^\infty h_2(u)(\tau-1)u^{-\tau} du\right)
$$

with $\phi$ the standard normal density. The integral factorization is simply classical independence of the commuting selfadjoint operators $T_g$ and $T_W$. Furthermore, $T_s$ has a standard semicircle law and is freely independent of $\left(T_g, T_W\right)$.

In particular, when $W$ is degenerate, say $W\equiv 1$, then $\nu_{\tau}$ is given by the free additive convolution of semicircle and Gaussian law.

\end{theorem}

\subsection{Discussion and simulations.} 
\begin{itemize}
\item[(a)] We now briefly describe the main steps of the proof. 
\begin{itemize}
\item[1.] {\bf Gaussianisation.} In the first step, we show that replacing the Bernoulli entries with Gaussian entries having the same mean and variance results in empirical spectral distributions that are close.

\item[2.] {\bf Simplification of the variance profile.} In this step, we show that the variance profile can be simplified to $W_iW_j/\|i-j\|^{\alpha}$, effectively removing the truncation at $1$.

\item[3.] {\bf Truncation.} Here, we show that in the Gaussian matrix, the weights $W_i$ can be replaced by the truncated weights $W_i^m = W_i\one_{W_i \le m}$.

\item[4.] {\bf Decoupling the diagonal.} In this step, we show that the Laplacian can be viewed as the sum of two independent random matrices (conditionally on the weights). Thus, we replace the diagonal matrix $D_N$ with an independent copy $Y_N$, which has the same variance profile.

\item[5.] {\bf Moment method.} With truncated weights and decoupled matrices, we apply the moment method to show convergence of the empirical spectral distribution and identify the limiting moments. A key observation is that the limiting measure and the method are independent of $\alpha$, so the results remain valid even when $\alpha = 0$.

\item[6.] {\bf Identification of the limiting measure.} Finally, we first identify the limiting measure in the case of truncated weights. These are typically associated with bounded operators (except in the Gaussian case). We then use techniques from \cite{bercovici1993free} to remove the truncation and identify the limiting measure in the general case.
\end{itemize}

\item[(b)] We now present some simulations that illustrate how the proof outline aligns with a specific value of $\alpha$. In Figure \ref{fig:spectrum_cenLap}, we plot the eigenvalue distribution of the centred Laplacian matrix, with the parameter range $N=6000$, $\alpha=0.5$ and $\tau=4.1$. A crucial step in the proof of Theorem \ref{theorem:Laplacian limiting measure} require us to replace the Bernoulli entries with Gaussian entries with the same variance profile. Also in the Gaussian case, we can simplify the variance to the following form: 
\[
\frac{W_iW_j}{\|i-j\|^{\alpha}}
\]
for any $(i,j)^{\text{th}}$ entry. We compare the two spectra in Figure \ref{fig:spectrum_compare}. We also consider the Gaussianised Laplacian matrix with a decoupled diagonal, and in Section \ref{sec:final proof Laplacian}, we apply an idea used in \cite{CHMS2025}, where we take $\alpha=0$. We also compare the spectrum of this matrix to the original centred Laplacian in Figure \ref{fig:spectrum_compare}. We see that the spectra are quite similar.  

\begin{figure}[ht!]
 \centering
\includegraphics[scale=.5]{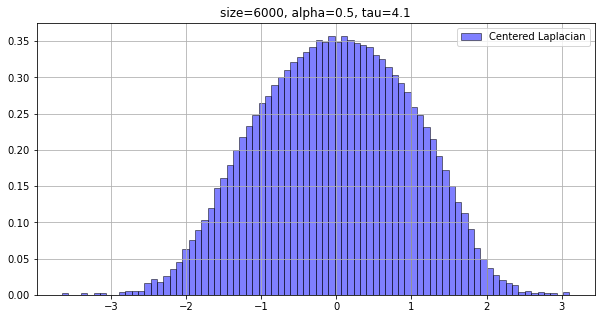}    
    \caption{Spectrum of the centred Laplacian matrix .}
\label{fig:spectrum_cenLap}
\end{figure}
\begin{figure}[ht!]
 \centering
\hfill
\subfloat{\includegraphics[width=0.5\linewidth]{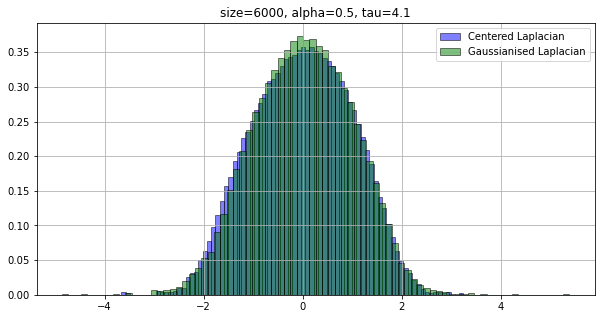}}
    \hfill    \subfloat{\includegraphics[width=0.5\linewidth]{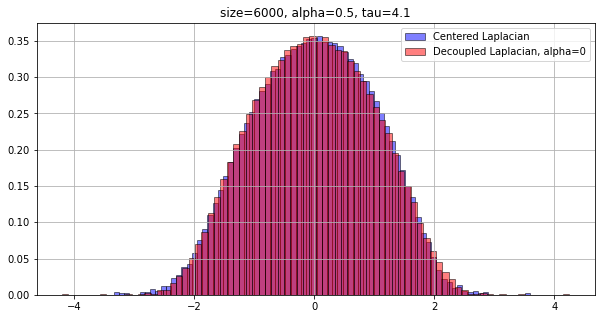}}
    \hfill
    \caption{Comparing the spectrum of the centred Laplacian with the Gaussianised and the decoupled case. }
    \label{fig:spectrum_compare}
\end{figure}

\item[(c)] We remark that our results can be extended in two directions. Although we state and prove them for the case $d = 1$ and $\alpha < 1$, they naturally generalize to any dimension $d \ge 1$ and $\alpha < d$. In that case, the scaling constant requires an adjustment, with $c_N \sim c_0(d)\, N^{d - \alpha}$. For ease of presentation, we restrict ourselves to $d = 1$ in this work.

Another possible extension of our first result involves modifying the connection probabilities between vertices $i$ and $j$ to
\[
 \frac{\kappa_{\sigma}(W_i, W_j)}{\|i - j\|^\alpha} \wedge 1,
\]
where $\kappa_{\sigma}(x, y) = (x \vee y)(x \wedge y)^{\sigma}$. In this setting, we additionally assume $0<\sigma < (\tau - 1)$. Such extensions have been studied in the context of adjacency matrices in~\cite{CHMS2025}. We strongly believe that in this case, the limiting spectral distribution will exist but it would be challenging to identify the limiting measure.

\end{itemize}

\subsection{Notation} 
We will use the Landau notation $o_N,\,\bigO_N$ indicating in the subscript the variable under which we take the asymptotic (typically this variable will grow to infinity unless otherwise specified). Universal positive constants are denoted as $c,\,c_1,\,\ldots,$ and their value may change with each occurrence. For an $N\times N$ matrix $\A=(a_{ij})_{i,\,j=1}^N$ we use $\Tr(\A)\coloneqq\sum_{i=1}^N a_{ii} $ 
for the trace and $\mathrm{tr}(\A)\coloneqq N^{-1}\Tr(\A)$ for the normalised trace. When $n\in\N$ we write $[n]\coloneqq\{1,\,2,\,\ldots,\,n\}.$ We denote the cardinality of a set $A$ as $\#A$, and, with a slight abuse of notation, $\#\sigma$ also denotes the number of cycles in a permutation $\sigma$.

\section{Gaussianisation: Setup for the proof of Theorem \ref{theorem:Laplacian limiting measure}}\label{section:Pre-moment method}

To prove Theorem \ref{theorem:Laplacian limiting measure}, we construct a Laplacian matrix with truncated weights and a simplified variance profile, with the diagonal decoupled from the adjacency matrix. We follow the ideas of \cite{CHMS2025}, albeit with a slightly modified approach, as follows:
\begin{enumerate}
    \item We begin by Gaussianising the matrix $\bfD_N^o$ to obtain a matrix $\bar{\bfD}_N$, using the ideas of \cite{chatterjee2005simple}. Since we have $\tau > 3$, the proof proceeds without the need to truncate the weight sequence $\{W_i\}_{i\in\Ver_N}$. 
    \item We then tweak the entries of $\bar{\bfD}_N$ further through a series of lemmas to obtain the Laplacian matrix $\widehat{\bfD}_{N,g}$, whose corresponding adjacency has mean-zero Gaussian entries and a simplified variance profile. 
    \item Next, we truncate the weights $\{W_i\}_{i\in\Ver_N}$ at $m \geq 1$, and construct the corresponding matrix $\bfD_{N,g,m}$. We show that, in $\prob$-probability, the L\'{e}vy distance vanishes in the iterated limit $m \to \infty$ and $N \to \infty$.
    \item We conclude by decoupling the diagonal of the matrix $\bfD_{N,g,m}$ from the off-diagonal terms. This follows from classical results used in studying the spectrum of Laplacian matrices. 
\end{enumerate}

\subsection{Gaussianisation}\label{subsec:Gaussianisation}

Suppose $(G_{i,j})_{i>j}$ is a sequence of i.~i.~d.\ $\gauss(0,1)$ random variables and independent of the sequence $(W_i)_{i\in \Ver_N}$. Define
$$\bar{\bA}_N(i,j)=\begin{cases} \frac{\sqrt{p_{ij}(1-p_{ij})} }{\sqrt{c_N}}G_{i\wedge j, i\vee j} +\frac{\mu_{ij}}{\sqrt{c_N}} &  i\neq j\\
0 & i=j,
\end{cases} 
$$
where $\mu_{ij} = p_{ij} - \E[p_{ij}]$. Let $\bar{\bfD}_N$ be the corresponding Laplacian of the matrix $\bar{\bA}_N$. Let $h$ be a $3$ times differentiable function on $\mathbb R$ such that $\max_{0\le k\le 3} \sup_{x\in \mathbb R} |h^{(k)}(x)|<\infty$, where $h^{(k)}$ is the $k$-th derivative of $h$. Define the resolvent of the $N\times N$ matrix $\M_N$ as
$$
R_{M_N}(z)=\left(\M_N-z \Id_N\right)^{-1},\qquad z \in \mathbb{C}^+,
$$
where $\Id_N$ is the $N\times N$ identity matrix and $\C^+$ is the upper-half complex plane. Further, define
$H_z(\M_N)= \St_{M_N}(z)= \tr(R_{M_N}(z))$ for $z\in \mathbb C^{+}$.  

\begin{lemma}[{\bf Gaussianisation of $\bfD_N$}]\label{lemma:Laplacian Gaussianisation}
Consider $\bar{\bfD}_N$ and $\bfD_N^o$ defined as above. Then for any $h$ as above,
$$\lim_{N\to\infty}\left| \E[ h(\Re H_z(\bar{\bfD}_N))]- \E[ h(\Re H_z(\bfD_{N}^o))]\right|=0\, ,$$
and 

$$\lim_{N\to\infty}\left| \E[ h(\Im H_z(\bar{\bfD}_N))]- \E[ h(\Im H_z(\bfD_{N}^o))]\right|=0\, .$$    
\end{lemma}

The proof is very similar to the one presented in \cite{Chatterjee:Hazra} and is modified along the lines of \cite{CHMS2025} and hence it is moved to Appendix \ref{appendix:prel_lemmas}. It uses the classical result of \cite{chatterjee2005simple}, and we only give a brief sketch by showing the estimates of the error probabilities in this setting. In \cite{CHMS2025}, the Gaussianisation was done with truncated weights, but here we will not need that. 

\subsection{\bf Simplification of the variance profile.}
We now proceed with a series of lemmas to simplify the variance profile of our Gaussianised matrix. First, we construct a new matrix $\bfD_{N,g}$ as the Laplacian corresponding to the matrix $\A_{N,g}$, defined as follows:

Suppose $(G_{i,j})_{i>j}$ is a sequence of i.~i.~d.\ $\gauss(0,1)$ random variables as before, and independent of the sequence $(W_i)_{i\in \Ver_N}$. Define
$$\A_{N,g}=\begin{cases} \frac{\sqrt{p_{ij}(1-p_{ij})} }{\sqrt{c_N}}G_{i\wedge j, i\vee j} &  i\neq j\\
0 & i=j.
\end{cases}
$$
We now have the following result. 
\begin{lemma}\label{lemma:mean zero gaussianisation}
Let $\bar{\bfD}_N$ and $\bfD_{N,g}$ be as defined above. Then, 
$$
\lim_{N\to\infty}\prob(d_L(\ESD(\bar{\bfD}_N),\ESD(\bfD_{N,g}))>\vep)=0.
$$
\end{lemma}

\begin{proof}[Proof of Lemma \ref{lemma:mean zero gaussianisation}]
By the Hoffman--Wielandt inequality (Proposition~ \ref{prop:hoffman-wielandt}),
\begin{align*}
\E\Big[d_L^3\big(\ESD(\bfD_{N,g}),\ESD(\bar\bfD_N)\big)\Big]
&\le \frac1N\,\E\Tr\Big((\bfD_{N,g}-\bar\bfD_N)^2\Big)\\
&=\frac1N\sum_{i\ne j}\E\left[(\A_{N,g}(i,j)-\bar \A_N(i,j))^2\right]\\
&+\frac1N\sum_{i=1}^N \E\left[\Big(\sum_{k\ne i}(\A_{N,g}(i,k)-\bar \A_N(i,k))\Big)^2\right].
\end{align*}
Recall that for $i\ne j$,
\[
\bar \A_N(i,j)=\frac{\sqrt{p_{ij}(1-p_{ij})}}{\sqrt{c_N}}\;G_{i\wedge j,i\vee j}+\frac{\mu_{ij}}{\sqrt{c_N}},
\qquad
\A_{N,g}(i,j)=\frac{\sqrt{p_{ij}(1-p_{ij})}}{\sqrt{c_N}}\;G_{i\wedge j,i\vee j},
\]
with $\mu_{ij}:=p_{ij}-\E[p_{ij}]$ and $c_N$ as in \eqref{chap5:eq:def_c}. Hence
\[
\A_{N,g}(i,j)-\bar \A_N(i,j)=-\frac{\mu_{ij}}{\sqrt{c_N}}.
\]

\smallskip
\noindent\textbf{Off–diagonal term.}
Using $\E[\mu_{ij}^2]\le \E[p_{ij}^2]$ and the bound $\E[p_{ij}^2]\leq C \|i-j\|^{-2\alpha}$ (from $\E[W^2]<\infty$), we get
\begin{align*}
\frac1N\sum_{i\ne j}\E\left[(\A_{N,g}(i,j)-\bar \A_N(i,j))^2\right]
=\frac{1}{Nc_N}\sum_{i\ne j}\E[\mu_{ij}^2]
\leq C \frac{1}{Nc_N}\sum_{i\ne j}\frac{1}{\|i-j\|^{2\alpha}}.
\end{align*}
By Lemma~\ref{lem:sumtoint} with $\beta=2\alpha$,
\[
\frac{1}{N}\sum_{i\ne j}\frac{1}{\|i-j\|^{2\alpha}}
\asymp
\begin{cases}
N^{1-2\alpha}, & 0<\alpha<\tfrac12,\\[2pt]
\log N, & \alpha=\tfrac12,\\[2pt]
1, & \tfrac12<\alpha<1,
\end{cases}
\]
while $c_N\asymp N^{1-\alpha}$ for $0<\alpha<1$. Therefore
\begin{equation}\label{eq:boundsalpha}
\frac1N\sum_{i\ne j}\E\left[(A_{N,g}(i,j)-\bar A_N(i,j))^2\right]
\leq C
\begin{cases}
N^{-\alpha}, & 0<\alpha<\tfrac12,\\[2pt]
(\log N)\,N^{-1/2}, & \alpha=\tfrac12,\\[2pt]
N^{-(1-\alpha)}, & \tfrac12<\alpha<1,
\end{cases}
\end{equation}
and clearly, all the cases go to zero.

\noindent\textbf{Diagonal (row–sum) term.}
Set $S_i:=\sum_{k\ne i}\mu_{ik}$, so that
\[
\sum_{k\ne i}(A_{N,g}(i,k)-\bar A_N(i,k))
=-\frac{1}{\sqrt{c_N}}\,S_i.
\]
Thus
\[
\frac1N\sum_{i=1}^N \E\!\left[\Big(\sum_{k\ne i}(A_{N,g}(i,k)-\bar A_N(i,k))\Big)^2\right]
= \frac{1}{Nc_N}\sum_{i=1}^N \E\big[S_i^2\big].
\]
We bound $\E[S_i^2]=\Var\!\big(\sum_{k\ne i}p_{ik}\big)$ by the Efron–Stein inequality (\cite[Theorem 3.1]{Boucheron:Lugosi:Massart}) applied to the independent family $\{W_k:k\ne i\}$ (with $W_i$ kept fixed as a parameter):
\[
\Var\!\Big(\sum_{k\ne i}p_{ik}\Big)\ \le\ \sum_{k\ne i}\E\Big[\big(p_{ik}(W_i,W_k)-p_{ik}(W_i,W_k')\big)^2\Big]
\ \le\ 2\sum_{k\ne i}\E\big[p_{ik}^2\big],
\]
where $W_k'$ is an independent copy of $W_k$. Using again $\E[p_{ik}^2]\leq C \|i-k\|^{-2\alpha}$, we obtain
\[
\frac{1}{Nc_N}\sum_{i=1}^N \E[S_i^2]
\ \leq \ \frac{C}{Nc_N}\sum_{i\ne k}\frac{1}{\|i-k\|^{2\alpha}},
\]
which is the same quantity already handled in the off–diagonal term. Hence it also goes to $0$ at the rates displayed above.

\smallskip
Combining the two bounds gives
\[
\E\Big[d_L^3\big(\ESD(\Delta_{N,g}),\ESD(\bar\Delta_N)\big)\Big]\longrightarrow 0,
\]
and the claim follows by Markov's inequality.
\end{proof}

Define for $i\neq j$
\[
\tilde{\A}_{N,g}(i,j)= \frac{\sqrt{p_{ij}}}{\sqrt{c_N}}G_{i\wedge j, i\vee j}
\]
and put zero on the diagonal. Here $(G_{i,j})_{i\ge j}$ are the i.~i.~d. $\gauss(0,1)$ random variables used in the previous result. Let $\tilde{\bfD}_{N,g}$ be analogously defined. 
The next lemma shows that $\bfD_{N,g}$ and $\tilde\bfD_{N,g}$ have asymptotically the same spectrum.

\begin{lemma}\label{lemma:gaussianisation without 1-}
$$\lim_{N\to\infty}\prob\left( d_L(\ESD(\tilde{\bfD}_{N,g}), \ESD(\bfD_{N,g}))>\varepsilon\right) =0.$$
\end{lemma}

\begin{proof}
    Again using Propostion \ref{prop:hoffman-wielandt}, we have that 
    \begin{align*}
   \E\left[d_L^3(\ESD(\tilde{\bfD}_{N,g}), \ESD(\bfD_{N,g}))\right] &\le \frac{1}{N}\E\Tr\left( (\bfD_{N,g}- \tilde{\bfD}_{N,g})^2\right)\\ 
&=\frac{1}{N}\E\left[\sum_{1\le i,j\le N}\left(\bfD_{N,g}(i,j)-\tilde{\bfD}_{N,g}(i,j)\right)^2\right]\\
&=\frac{1}{N}\sum_{1\le i\neq j\le N}\E\left[ (\A_{N,g}(i,j)-\tilde{\A}_{N,g})^2\right]\\
&\qquad + \frac{1}{N}\sum_{i=1}^N \E\left[\left(\sum_{k\neq i} \A_{N,g}(i,k)-\tilde{\A}_{N,g}(i,k)\right)^2\right]
    \end{align*}

    Dealing with the last two terms separately as before, we proceed by bounding the first term. Since $G_{ij}\sim\mathcal N(0,1)$ are independent of $(W_i)$
and of each other,
$$ \E\left[\left(\A_{N, g}(i, j)-\widetilde{\A}_{N, g}(i, j)\right)^2\right]=\frac{1}{c_N}  \E\left[\left(\sqrt{p_{i j}}-\sqrt{p_{i j}\left(1-p_{i j}\right)}\right)^2\right]=\frac{1}{c_N} \E\left[p_{i j}\left(1-\sqrt{1-p_{i j}}\right)^2\right] .$$

Using $0\le p\le 1$ and the elementary bound $0\le 1-\sqrt{1-p}\le p$ we have $\big(1-\sqrt{1-p}\,\big)^2\le
p^2$, hence \begin{equation} \E\Big[(\A_{N,g}(i,j)-\widetilde \A_{N,g}(i,j))^2\Big] \ \le\ \frac1{c_N}\,\E\big[p_{ij}
^2\big]  \le \frac{C}{c_N}\,\frac{\mathbf E[W_1^2]^2}{\|i-j\|^{2\alpha}}, \label{eq:I1-ptwise} 
\end{equation} Summing
\eqref{eq:I1-ptwise} over $i\ne j$ and applying Lemma~\ref{lem:sumtoint} and using the bounds in \eqref{eq:boundsalpha} we have that the term goes to zero.


    Expanding the square in the second term, we have
\begin{align*}
   & \frac{1}{N}\sum_{i=1}^N \E\left[\left(\sum_{k\neq i} \A_{N,g}(i,k)-\tilde{\A}_{N,g}(i,k)\right)^2\right]= \frac{1}{N}\sum_{i=1}^N \sum_{k\neq i} \E\left[\left(\A_{N,g}(i,k)-\tilde{\A}_{N,g}(i,k)\right)^2\right]\\
    &+\frac{1}{N}\sum_{i=1}^N\sum_{k\neq i}\sum_{\ell\neq i, k}\E\left[\left( \A_{N,g}(i,k)-\tilde{\A}_{N,g}(i,k)\right)\left( \A_{N,g}(i,\ell)-\tilde{\A}_{N,g}(i,\ell)\right)\right].
\end{align*} 

Again the first term in above sum is of the order $N^{-\alpha}$ and the expectation in the second term is zero. Indeed, using the independence between $(W_i)_{i\in \Ver_N}$ and $G_{i,j}$ we have for $k\neq \ell$,
\begin{align*}
&\E\left[\left( \A_{N,g}(i,k)-\tilde{\A}_{N,g}(i,k)\right)\left( \A_{N,g}(i,\ell)-\tilde{\A}_{N,g}(i,\ell)\right)\right]\\
&= \E\left[ (\sqrt{p_{ik}(1-p_{ik})}-\sqrt{p_{ik}})(\sqrt{p_{i\ell}(1-p_{i\ell})}-\sqrt{p_{i\ell}})\right]\E\left[G_{i \wedge k, i \vee k} G_{i \wedge \ell, i \vee \ell}\right]=0,
\end{align*}
where the last term is zero for $\ell \notin \{i,k\}$ due to independence. This completes the proof of the Lemma.
\end{proof}

We conclude this subsection with one final simplification. For any $i\neq j$, let
$$r_{ij}= \frac{W_i W_j}{\|i-j\|^{\alpha}}\, ,$$ and let $r_{ii}=0$. Define the matrix $\widehat{\A}_{N,g}$ as follows: for $i\neq j$, 
\[
\widehat{\A}_{N,g}(i,j)= \frac{\sqrt{r_{i\wedge j, i\vee j}}}{\sqrt{c_N}} G_{i\wedge j, i\vee j}\]
and put $0$ on the diagonal. Define Laplacian matrix $\widehat{\bfD}_{N,g}$ accordingly with $\widehat{\A}_{N,g}$.

\begin{lemma}\label{lemma:removal of wedge 1}
  $$\lim_{N\to\infty}\prob\left( d_L(\ESD(\tilde{\bfD}_{N,g}), \ESD(\widehat{\bfD}_{N,g}))>\varepsilon\right) =0.$$  
\end{lemma}

\begin{proof}
For any $1\le i\neq j\le N$, define the set $\mathcal C_{ij}=\{r_{ij}<1\}$. Let $(X_{i,j})_{i\ge j}$ be defined as follows
$$X_{ij}= \frac{\sqrt{r_{ij}}}{\sqrt{c_N}} G^{\prime}_{ij},$$
where $(G^{\prime}_{ij})_{i\ge j}$ be a sequence of independent $\gauss(0,1)$ random variables, independent of the previously defined $(G_{ij})$ and $(W_i)_{i\in \Ver_N}$. Define a symmetric matrix $L_{N,g}$ as follows: for $1\le i<j\le N$,
$$L_{N,g}(i,j)= \tilde{\A}_{N,g}(i,j)\one_{\mathcal C_{ij}}+ X_{ij} \one_{\mathcal C_{ij}^c}.$$
We put zero on the diagonal and consider the $\bfD_{L}$ as the Laplacian matrix corresponding to $L_{N,g}$. Note that $L_{N,g}$ has the same distribution as $\widehat{\A}_{N,g}$ and hence the $\bfD_{L}$ has the same distribution as $\widehat{\bfD}_{N,g}$. 

By Proposition \ref{prop:hoffman-wielandt}, we again have
\begin{align*}
&\E\left[d^3\left(\ESD(\bfD_{L}), \ESD(\tilde{\bfD}_{N,g})\right)\right] 
\le \frac{1}{N}\E\left[\sum_{1\le i,j \le N} \left(\bfD_L(i,j)-\tilde{\bfD}_{N,g}(i,j)\right)^2\right]\\
&=\frac{1}{N}\sum_{1\le i\neq j\le N}\E\left[ (L_{N,g}(i,j)-\tilde{\A}_{N,g}(i,j))^2\right]\\
&\qquad + \frac{1}{N}\sum_{i=1}^N \E\left[\left(\sum_{k\neq i} L_{N,g}(i,k)-\tilde{\A}_{N,g}(i,k)\right)^2\right]\\
&\le \frac{4}{N}\sum_{1\le i\neq j\le N}\E\left[ (L_{N,g}(i,j)-\tilde{\A}_{N,g}(i,j))^2\right]\\
    &+\frac{1}{N}\sum_{i=1}^N\sum_{k\neq i}\sum_{\ell\neq i, k}\E\left[\left( L_{N,g}(i,k)-\tilde{\A}_{N,g}(i,k)\right)\left( L_{N,g}(i,\ell)-\tilde{\A}_{N,g}(i,\ell)\right)\right]
\end{align*}

Again, we deal with the two sums separately. The first sum can be bounded above as follows:
\begin{align*}
 \frac{4}{N}\sum_{1\le i\neq j\le N}\E\left[ (L_{N,g}(i,j)-\tilde{\A}_{N,g}(i,j))^2\right]
 &\le \frac{4}{N}\sum_{1\le i\neq j\le N}\E\left[ (\tilde{\A}_{N,g}(i,j)- X_{ij})^2\one_{\mathcal C_{ij}}\right]\\
 &\le \frac{8}{N}\sum_{1\le i\neq j\le N}\E\left[\tilde{\A}_{N,g}(i,j)^2\one_{\mathcal C_{ij}}\right]+ \E\left[X_{ij}^2\one_{\mathcal C_{ij}}\right]\\
 &\le\frac{1}{Nc_N}\sum_{i\neq j\in\Ver_N}^N \ep[G_{i\wedge j,i\vee j}^2\one_{\mathcal C_{ij}^c}]+\ep[X_{ij}^2\one_{\mathcal C_{ij}^c}]  \\
        &\leq \frac{1}{Nc_N}\sum_{i\neq j\in\Ver_N}^N \pr(\mathcal C_{ij}^c)+ \ep[X_{ij}^4]^{1/2}\pr(\mathcal C_{ij}^c)^{1/2} \\
        &\leq \frac{1}{Nc_N}\sum_{i\neq j\in\Ver_N}^N \pr(\mathcal C_{ij}^c)+\frac{3\ep[W_i^2 W_j^2]^{1/2}}{\|i-j\|^{\alpha}}\,
        \pr(\mathcal C_{ij}^c)^{1/2}\\
        &\le C( N^{-\alpha(\tau-2)}+ N^{-\frac{\alpha}{2}(\tau-1)})=\mathrm{o}_N(1),
\end{align*}
where we have used in the last line the following estimate:
\[
\pr(\mathcal C_{ij}^c) \leq \pr\left(W_iW_j\geq \|i-j\|^\alpha\right) 
    \leq  \frac{c}{\|i-j\|^{\alpha(\tau-1)}}
    \]
which follows from Lemma \ref{lemma:twotails}. For the second term note that

\begin{align*}
    &\E\left[\left( L_{N,g}(i,k)-\tilde{\A}_{N,g}(i,k)\right)\left( L_{N,g}(i,\ell)-\tilde{\A}_{N,g}(i,\ell)\right)\right]= \frac{1}{c_N}\mathbf{E}[\sqrt{p_{ik}}\sqrt{p_{i\ell}} \one_{\mathcal C_{ij}^c}\one_{\mathcal C_{i\ell}^c}]\E[G_{ik}G_{i\ell}]\\
    &\qquad -\frac{1}{c_N}\ep[\sqrt{p_{ik}}\sqrt{r_{i\ell}} \one_{\mathcal C_{ij}^c}\one_{\mathcal C_{i\ell}^c}]\E[G_{ik}G_{i\ell}^{\prime}]-\frac{1}{c_N}\ep[\sqrt{r_{ik}}\sqrt{p_{i\ell}} \one_{C_{ij}^c}\one_{C_{i\ell}^c}]\E[G_{ik}^{\prime} G_{i\ell}]\\
    &\qquad +\frac{1}{c_N}\ep[\sqrt{r_{ik}}\sqrt{r_{i\ell}} \one_{C_{ij}^c}\one_{C_{i\ell}^c}]\E[G_{ik}^{\prime} G_{i\ell}^{\prime}],
\end{align*}
and since $k\neq \ell$, all the above terms are zero. Thus the proof follows.
\end{proof}

\subsection{Truncation}
Let $m>1$ be a truncation threshold and define $W_i^m= W_i\one_{W_i\le m}$ for any $i\in\Ver_N$. For all $N\in\N$, we define a new random matrix as follows: 
Let 
\[
r_{ij}^m
    = \frac{W_i^m W_j^m}{\|i-j\|^\alpha}\qquad i\neq j\in \Ver_N\,,
\]
and let $\A_{N,g,m}$ be defined for $i \neq j$ as
$$\A_{N,g,m}(i,j)= \frac{\sqrt{r_{ij}^m}}{\sqrt{c_N}}G_{i\wedge j, i\vee j},$$
and put $0$ on the diagonal. Analogously define $\bfD_{N,g,m}$.

\begin{lemma}[{Truncation}]\label{lemma:truncation}
   For every $\delta>0$ one has 
   $$
   \limsup_{m\to \infty}\lim_{N\to\infty}\prob\left(d_L(\ESD(\bfD_{N, g,m}), \ESD(\tilde{\bfD}_{N,g}))> \delta\right)=0\,.
   $$
\end{lemma}

\begin{proof}
    The proof follows the same idea as the previous lemmas. Recall that $$\widehat{\bA}_{N,g}(i,j) = \frac{\sqrt{r_{ij}}}{\sqrt{c_N}}G_{i\wedge j,i\vee j}$$ for all $i\neq j$, with 0 on the diagonal, and $\widehat{\bfD}_{N,g}$ is the corresponding Laplacian. Once again, we have 
   \begin{align*}
&\E\left[d^3\left(\ESD(\bfD_{N,g,m}), \ESD(\widehat{\bfD}_{N,g})\right)\right] \\
&\le \frac{1}{N}\E\left[\sum_{1\le i,j \le N} \left(\bfD_{N,g,m}(i,j)-\widehat{\bfD}_{N,g}(i,j)\right)^2\right]\\
&=\frac{1}{N}\sum_{1\le i\neq j\le N}\E\left[ (\A_{N,g,m}(i,j)-\widehat{\A}_{N,g}(i,j))^2\right]\\
&\qquad + \frac{1}{N}\sum_{i=1}^N \E\left[\left(\sum_{k\neq i} \A_{N,g,m}(i,k)-\widehat{\A}_{N,g}(i,k)\right)^2\right]\\
&\le \frac{4}{N}\sum_{1\le i\neq j\le N}\E\left[ (\A_{N,g,m}(i,j)-\widehat{\A}_{N,g}(i,j))^2\right]\\
    &+\frac{1}{N}\sum_{i=1}^N\sum_{k\neq i}\sum_{\ell\neq i, k}\E\left[\left( \A_{N,g,m}(i,k)-\widehat{\A}_{N,g}(i,k)\right)\left( \A_{N,g,m}(i,\ell)-\widehat{\A}_{N,g}(i,\ell)\right)\right].
\end{align*}
The proof of Lemma \ref{lemma:removal of wedge 1} aids us by taking care of the second factor in the last line, which turns out to be equal to 0 by the independence of Gaussian terms. For the first term, the common Gaussian factor pulls out by independence, yielding the upper bound 
$$
\frac{4}{Nc_N}\sum_{1\leq i\neq j\leq N} \frac{\mathbf{E}\left[\left(\sqrt{W_iW_j} - \sqrt{W_i^mW_j^m}\right)^2\right]}{\|i-j\|^{\alpha}} \le \frac{4}{Nc_N}\sum_{1\leq i\neq j\leq N} \frac{\mathbf{E}[W_iW_j - W_i^mW_j^m]}{\|i-j\|^{\alpha}},$$
where the inequality follows by using the identity $(a-b)^2\le |a^2-b^2|$ for any $a,b\geq 1$. Adding and subtracting the term $W_iW_j^m$ inside the expectation gives us that 
\begin{align*}
&\frac{4}{N}\sum_{1\le i\neq j\le N}\E\left[ (\A_{N,g,m}(i,j)-\widehat{\A}_{N,g}(i,j))^2\right]\\
&\le \frac{4}{Nc_N}\sum_{1\le i\neq j\le N} \frac{\mathbf{E}[W_i]\mathbf{E}[W_j\one_{\{W_j>m\}}] + \mathbf{E}[W_j^m]\mathbf{E}[W_i\one_{\{W_i>m\}}]}{\|i-j\|^{\alpha}} \\
&\le \frac{C_{\tau}}{Nc_N}\sum_{1\leq i\neq j\le N} \frac{m^{2-\tau}}{\|i-j\|^{\alpha}} = \bigO_m(m^{2-\tau}),
\end{align*}
where the last inequality follows from Lemma \ref{lem:Par_trunc}, with $C_\tau$ a $\tau-$dependent constant. Markov inequality concludes the proof. 
\end{proof}

\subsection{Decoupling}
Since we now have bounded weights, the decoupling result follows from the arguments from \cite[Lemma 4.12]{BDJ2006}. See also the proof of \cite[Lemma 2.4]{CHHS} for the inhomogeneous extension.
\begin{lemma}\label{lemma:decoupling} Let $(Z_i : i \geq 1)$ be a family of i.~i.~d. standard normal random variables, independent of $(G_{i,j} : 1 \leq i \leq j)$. Define a diagonal matrix $Y_N$ of order $N$ by
\[
Y_N(i,i) = Z_i \sqrt{\frac{\sum_{k\neq i} r_{ik}^m}{c_N}}, \quad 1 \leq i \leq N.
\]

and let
\begin{equation}\label{eq:decoupled Laplacian}
\bfD_{N,g,c} = {\A}_{N,g,m} + Y_N. 
\end{equation}

Then for every $m>1$, and for any $k \in \mathbb{N}$,
\[
\lim_{N \to \infty} \frac{1}{N} \E \left( \operatorname{Tr} \left[ (\bfD_{N,g,c})^{2k} - (\bfD_{N,g,m})^{2k} \right] \right) = 0.
\]
\textit{and}
\[
\lim_{N \to \infty} \frac{1}{N^2} \E \left( \operatorname{Tr}^2 \left[ (\bfD_{N,g,c})^k \right] - \operatorname{Tr}^2 \left[ (\bfD_{N,g,m})^k \right] \right) = 0.
\]
\end{lemma}




\section{Moment method: Proof of Theorem \ref{theorem:Laplacian limiting measure}}\label{sec:moment method Laplacian}
We begin by stating a key proposition that describes the limit of the empirical spectral distribution of $\bfD_{N,g,c}$. The majority of this section will be devoted to the proof of this proposition, and so, we defer the proof of the proposition to page \pageref{proof:proof of prop}. 

\begin{proposition}\label{main:prop}
Let $\ESD(\bfD_{N,g,c})$ be the empirical spectral distribution of $\bfD_{N,g,c}$ defined in \eqref{eq:decoupled Laplacian}. Then there exists a deterministic measure $\nu_{\tau}$ on $\mathbb R$ such that
    $$
    \lim_{N\to \infty}\ESD(\bfD_{N,g,c})=\nu_{\tau,m}\qquad\text{ in $\,\prob$--probability}\,.
    $$
\end{proposition}

We now use Proposition \ref{main:prop} and tools from Appendix \ref{appendix:prel_lemmas} and Section \ref{section:Pre-moment method} to prove Theorem \ref{theorem:Laplacian limiting measure}. 

\begin{proof}[Proof of Theorem \ref{theorem:Laplacian limiting measure}]
   Combining Proposition \ref{main:prop} with Lemma \ref{lemma:decoupling} gives us that 
   \begin{align}\label{eq:Laplacian proof step 1} 
    \lim_{N\to \infty}\ESD(\bfD_{N,g,m})=\nu_{\tau,m}\qquad\text{ in $\,\prob$--probability}\,.
   \end{align}
   To show the existence of the limit $\nu_\tau :=\lim_{m\to\infty}\nu_{\tau,m}$, we wish to apply Lemma \ref{lemma:slutsky}. Equation \eqref{eq:Laplacian proof step 1} satisfies Condition \ref{item:1ar} of Lemma \ref{lemma:slutsky}. Moreover, Condition \ref{item:2ar} can be easily verified by Lemma \ref{lemma:truncation}. Thus, there exists a unique limit $\nu_\tau$ such that 
   \begin{align}\label{eq:Laplacian proof step 2}
       \lim_{N\to \infty}\ESD(\tilde{\bfD}_{N,g})=\nu_{\tau}\qquad\text{ in $\,\prob$--probability}\,.
   \end{align}
   Combining equation \eqref{eq:Laplacian proof step 2} with Lemma \ref{lemma:removal of wedge 1}, and subsequently with Lemma \ref{lemma:gaussianisation without 1-} and Lemma \ref{lemma:mean zero gaussianisation} yields 
   \begin{align}\label{eq:Laplacian proof step 3}
       \lim_{N\to \infty}\ESD(\bar{\bfD}_{N})=\nu_{\tau}\qquad\text{ in $\,\prob$--probability}\,.
   \end{align}
   We now wish to show that the limiting empirical spectral distribution for $\bfD_N^\circ$ is $\nu_\tau$ in $\prob$--probability. To this end, note that for any $h$ satisfying conditions of Lemma \ref{lemma:Laplacian Gaussianisation}, and $H_z$ as in subsection \ref{subsec:Gaussianisation}, we have by the means of Lemma \ref{lemma:Laplacian Gaussianisation} that 
   $$
\lim_{N\to\infty}h\left(\Re(H_z(\bfD_N^\circ))\right)= h\left(\Re\St_{\nu_\tau}(z)\right).
   $$
   The above characterises convergence in law. However, since $\nu_\tau$ is a deterministic measure, the above convergence holds in $\prob$--probability, and analogously for $\Im(H_z(\bfD_N^\circ))$. This gives us that $$
   \lim_{N\to\infty}\St_{\ESD(\bfD_N^\circ)}(z) = \St_{\nu_\tau}(z) \qquad\text{ in $\,\prob$--probability}\,.$$
   Since convergence of Stieltjes transforms characterises weak convergence, we obtain 
   $$
   \lim_{N\to\infty}\ESD(\bfD_N^\circ) = \nu_\tau \qquad\text{ in $\,\prob$--probability}\, ,$$ completing the proof. 
\end{proof}

We now provide the proof of Proposition \ref{main:prop}. We borrow the main ideas of \citet[Section 5.2.1, 5.2.2]{Chatterjee:Hazra}, and adapt them to our setting using the results of \citet[Section 4.4]{CHMS2025}. 

\begin{proof}[Proof of Proposition \ref{main:prop}]\label{proof:proof of prop}
    The proof of the moment method is valid when the weights are bounded, and so for notational convenience, in this proof we will drop the dependence on $m$ from $\{r^m_{ij}\}_{i,j\in\Ver_N}$. Thus, for the remainder of the proof, we have that $$
    r_{ij} = \frac{W_i^mW_j^m}{\|i-j\|^{\alpha}}.
    $$
    We apply the method of moments to show the convergence to the law $\nu_{\tau,m}$. The proof is split up into three parts as follows:
    \begin{enumerate}
        \item For any $k\geq 1$, we compute the expected moment $$\E\int_1^{\infty} x^k \ESD(\bfD_{N,g,c})(\dd{x}),$$
        and show that as $N\to\infty$, the above quantity converges to a value $0<M_k<\infty$ for $k$ even, and 0 otherwise. 
        \item We then show concentration by proving (under the law $\prob$) that $$\Var \left( \int_1^{\infty} x^k \ESD(\bfD_{N,g,c})(\dd{x})   \right) \to 0 \quad \text{as $N\to\infty$}.$$
        \item Lastly, we show that the sequence $\{M_k\}_{k\geq 1}$ uniquely determises a limiting measure. 
    \end{enumerate}

  \paragraph{Step 1.}  We begin by considering that $k$ is even. By using the expansion for $(a+b)^k$, it is easy to see that $$
    \E\int_1^{\infty} x^k \ESD(\bfD_{N,g,c})(\dd{x}) = \frac{1}{N}\E\left[\Tr\left(\bfD_{N,g,c}^k\right)\right] = \frac{1}{N}\sum_{\substack{m_1,\ldots,m_k,\\n_1,\ldots,n_k}}\E\left[\Tr\left(\A_{N,g,m}^{m_1}Y_N^{n_1}\ldots\A_{N,g,m}^{m_k}Y_N^{n_k}\right) \right],
    $$
    where $\A_{N,g,m}$ and $Y_N$ are as in Lemma \ref{lemma:decoupling}, and $\{m_i, n_i\}_{1\leq i \leq k}$ are such that $\sum_{i=1}^k m_i+n_i = k.$ 
    
    \noindent Let $M(p)$ and $N(p)$ be defined as 
    $$ M(p) = \sum_{i=1}^p m_i, \quad N(p)=\sum_{i=1}^p n_i$$ for any $1\leq p\leq k$. To expand the trace term, we sum over all $\mathbf{i} = (i_1,\ldots,i_{M(k)+N(k)+1})\in [N]^{M(k)+N(k)+1}$, where $[p]:=\{1,2,\ldots,p\}$, and we identify $i_{M(k)+N(k)+1} \equiv i_1$. Then, from \citet[Eq. 5.2.2]{Chatterjee:Hazra}, we have 
    \begin{align}\label{eq:Laplacian polynomial expansion}
        &\frac{1}{N}\Tr\left( \A_{N,g,m}^{m_1}Y_N^{n_1}\ldots\A_{N,g,m}^{m_k}Y_N^{n_k}\right) \nonumber \\
        &= \frac{1}{N}\sum_{i_1,\ldots,i_{M(k)}=1}^N\prod_{j=1}^{M(k)} G_{i_j \wedge i_{j+1}, i_j\vee i_{j+1}}\prod_{j=1}^{M(k)}\frac{\sqrt{r_{i_ji_{j+1}}}}{\sqrt{c_N}} \prod_{j=1}^k \left( \frac{1}{c_N}\sum_{t=1}^N r_{i_{1+M(j)}t} \right)^{\frac{n_j}{2}}\prod_{j=1}^k Z_{i_{1+M(j)}}^{n_j} \, ,
    \end{align}
    where also in \eqref{eq:Laplacian polynomial expansion} we identify $i_{M(k)+1}\equiv i_1$. 
    Taking expectation in \eqref{eq:Laplacian polynomial expansion}, we have that 
    \begin{align}\label{eq:Laplacian moment expectation expansion}
    &\E\left[\tr\left( \A_{N,g,m}^{m_1}Y_N^{n_1}\ldots\A_{N,g,m}^{m_k}Y_N^{n_k}\right)\right] \nonumber \\
    &= \frac{1}{N}\sum_{i_1,\ldots,i_{M(k)}} \E\left[ \prod_{j=1}^{M(k)} G_{i_j \wedge i_{j+1}, i_j\vee i_{j+1}}\right]\E\left[\prod_{j=1}^{M(k)}\frac{\sqrt{r_{i_ji_{j+1}}}}{\sqrt{c_N}} \prod_{j=1}^k \left( \frac{1}{c_N}\sum_{t=1}^N r_{i_{1+M(j)}t} \right)^{\frac{n_j}{2}}\right]\E\left[\prod_{j=1}^k Z_{i_{1+M(j)}}^{n_j}     \right].
    \end{align}
     It is well known that the expectation over a product of independent Gaussian random variables is simplified using the Wick's formula (see Lemma \ref{lemma:Wick}). In particular, if one were to partition the tuple $\{1,\ldots,K\}$ for some nonnegative integer $K$, the contributing partitions are typically non-crossing pair partitions (\cite{Nica:Speicher}). 
    
    We now introduce some notation from \cite{CHMS2025}. For any fixed non-negative even integer $K$, let $\mathcal{P}_2(K)$ and $NC_2(K)$ be the set of all pair partitions and the set of all non-crossing pair partitions of $[K]$ respectively. Let $\gamma = (1,\ldots, K)\in S_K$ be the right-shift permutation (modulo $K$), and for any $\pi$ which is a pair-partition, we identify it as a permutation of $[K]$, and read $\gamma\pi$ as a composition of permutations. Further, for any $\pi\in \mathcal{P}_2(K)$, let $\mathrm{Cat}_{\pi}$ denote the set $$
    \mathrm{Cat}_\pi := \mathrm{Cat}_\pi(K,N) = \{\mathbf{i}\in [N]^K : i_r = i_{\gamma\pi(r)} \text{ for all $r\in [K]$}\}. 
    $$
    Let $C(K,N)= \mathrm{Cat}_\pi^c$, the complement of $\mathrm{Cat}_{\pi}$, wherein we have $i_r=i_{\pi(r)}$ for any $r$. By Wick's formula for the Gaussian terms $\{G_{i,j}\}$, since the the sum over tuples $\mathbf{i}$ would be reduced to the sum over pair partitions $\pi\in\mathcal{P}_2(K)$ and the associated tuples $\mathbf{i}\in \mathrm{Cat}_\pi\cup C(K,N)$, we can write \begin{align}\label{eq:Laplacian sum split}
    \sum_{\mathbf{i}\in [K]^N}=\sum_{\pi \in \mathcal{P}_2(K)}\sum_{\mathbf{i}\in C(K,N)} + \sum_{\pi\in NC_2(K)}\sum_{\mathbf{i}\in\mathrm{Cat}_\pi} + \sum_{\pi\in \mathcal{P}_2(K)\setminus NC_2(K)}\sum_{\mathbf{i}\in\mathrm{Cat}_\pi}.\end{align}
    To analyse further, we use a key tool in the proof which is the following fact (\citet[Claim 4.10]{CHMS2025}). 
    \begin{fact}\label{fact:Laplacian-claim410}
        Let $K$ be an even nonnegative integer. Then, we have the following to be true: 
        \begin{enumerate}
            \item For any $\pi\in NC_2(K)$, we have 
            $$
            \lim_{N\to\infty}\frac{1}{Nc_N^{K/2}}\sum_{\mathbf{i}\in\mathrm{Cat}_\pi}\prod_{(r,s)\in\pi}\frac{1}{\|i_r - i_{r+1}\|^{\alpha}} = 1\, . 
            $$
            \item For any pair partition $\pi$, if $\mathbf{i}\in C(K,N)$, then, 
            $$
            \lim_{N\to\infty}\frac{1}{Nc_N^{K/2}}\sum_{\mathbf{i}\in C(K,N)}\prod_{(r,s)\in\pi}\frac{1}{\|i_r - i_{r+1}\|^{\alpha}} = 0\, . 
            $$
        \item For a partition $\pi\in \mathcal{P}_2(K)\setminus NC_2(K)$, we have 
        $$
        \lim_{N\to\infty}\frac{1}{Nc_N^{K/2}}\sum_{\mathbf{i}\in \mathrm{Cat}_\pi\cup C(K,N)}\prod_{(r,s)\in\pi}\frac{1}{\|i_r - i_{r+1}\|^{\alpha}} = 0\, . 
        $$
        \end{enumerate}
    \end{fact} 
    Let $\tilde{\pi}:=\gamma\pi$ for any choice of $\pi$. From \citet[Eq. 5.2.5]{chatterjee2005simple}, we have that 
    \begin{align}\label{eq:Laplacian moment beta term}
        \mathcal{E}(\tilde{\pi}) := \E\left[ \prod_{j=1}^k Z_{i_{1+M(j)}}^{n_j}\right] = \prod_{u\in\tilde{\pi}} \E \left[  \prod_{\substack{j\in[k]:\\ 1+M(j)\in u}} Z_{\ell_u}^{n_j} \right] < \infty ,
    \end{align}
    where $u$ is a block in $\tilde{\pi}$ and $\ell_u$ its representative element. Note that this does not depend on the choice of $\mathbf{i}$, and to obtain a nonzero contribution, we must have that for all $u\in\tilde{\pi}$, 
    \begin{align}\label{eq:Laplacian n_j condition}
        \sum_{\substack{j\in[k]: 1+M(j)\in u} } n_j \equiv 0 \quad (\mathrm{mod}\, 2).
    \end{align}
    Observe that $\E\left[ \prod_{j=1}^{M(k)} G_{i_j \wedge i_{j+1},i_j\vee i_{j+1}}\right]$ depends only on $\tilde{\pi}$ and not the choice of $\mathbf{i}$, and as a consequence, we can define 
    \begin{align}\label{eq:Laplacian moment phi term}
        \Phi(\tilde{\pi}) := \E\left[ \prod_{j=1}^{M(k)} G_{i_j \wedge i_{j+1},i_j\vee i_{j+1}}\right]<\infty \, .
    \end{align}
    Next, note that the sum
    \begin{align}\label{eq:Laplacian second r terms}
    \frac{1}{c_N}\sum_{t=1}^N r_{i_{1+M(j)}t} = \bigO_N(1)
    \end{align} by definition of $c_N$ (and the weights are uniformly bounded). Finally, if we look at the terms \begin{align}\label{eq:Laplacian_decaytermC(K,N)}
    \E\left[    \left(  \prod_{j=1}^{M(k)}\frac{r_{i_ji_{j+1}}}{c_N}    \right)^{1/2}    \right]\, ,
    \end{align} we can again bound the weights above by $m$. Recall that Wick's formula on the Gaussian terms imposes the restriction on choices of $\mathbf{i}$. Using these facts, in combination with \eqref{eq:Laplacian moment beta term}, \eqref{eq:Laplacian moment phi term}, and \eqref{eq:Laplacian second r terms}, we have that \eqref{eq:Laplacian moment expectation expansion} gets bounded by  \begin{align}\label{eq:Laplacian moment expectation bound}    
    \eqref{eq:Laplacian moment expectation expansion} \le \frac{C}{Nc_N}\sum_{\pi\in\mathcal{P}_2(M(k))}\sum_{\mathbf{i}\in \mathrm{Cat}_\pi\cup C(M(k),N)} \Phi(\tilde{\pi})\mathcal{E}(\tilde{\pi})\prod_{(r,s)\in\pi}\frac{m^2}{\|i_r - i_{r+1}\|^{\alpha}}\, . \end{align}
    If we split \eqref{eq:Laplacian moment expectation bound} as \eqref{eq:Laplacian sum split}, then using Fact \ref{fact:Laplacian-claim410}, we see that in the cases when $\pi\in\mathcal{P}_2(M(k))$ and $\mathbf{i}\in C(M(k),N)$, and when $\pi\in\mathcal{P}_2(M(k))\setminus NC_2(M(k))$ for all $\mathbf{i}$, the contribution in the limit $N\to\infty$ is 0.

    We are now in the setting where we take $\pi\in NC_2(M(k))$ and $\tilde{\pi}:=\gamma\pi$, and $\mathbf{i}\in \mathrm{Cat}_\pi$. First, note that $\tilde{\pi}$ is a partition of $[M(k)]$. We remark that if $M(k)\equiv 1 \, (\mathrm{mod}\,2)$ then $NC_2(M(k))=\emptyset$, and so, $M(k)$ must be even. 
    
    Next, we focus on analysing the product $\prod_{j=1}^{M(k)}\sqrt{r_{i_ji_{j+1}}^m}$ appearing in \eqref{eq:Laplacian moment expectation expansion}. We wish to show that this depends only on $\pi$, and not on the choice of $\mathbf{i}$. We follow the idea of \cite{CHMS2025}, wherein one constructs a graph associated to a chosen partition $\pi$, and any tuple $\mathbf{i}\in \mathrm{Cat}_\pi$ is equivalent to a tuple $\tilde{\mathbf{i}}$ with as many distinct indices as the number of vertices in the constructed graph. First, note that the coordinates are pairwise distinct (we take $r_{ii}=0$ for all $i$). Next, we construct a preliminary graph from the closed walk $i_1\to i_2\to\ldots i_{M(k)}\to i_1$. Lastly, we collapse vertices and edges that are matched in $\mathrm{Cat}_\pi$, and we denote the resulting graph as $G_{\tilde{\pi}}$, since it does not depend on the choice of $\mathbf{i}$ but rather the choice of $\pi$ itself. The resulting graph $G_{\tilde{\pi}}$ is the graph associated to the partition $\pi$, and we refer the reader to Definition \ref{pap1-partitiongraph} for a formal description. For clarity, consider the following example: 
    
    Let $M(k) = 4$, and let $\pi=\{\{1,2\},\{3,4\}\}$. Then, $\tilde{\pi}=\{\{1,3\},\{2\},\{4\}\}$. For any $\mathbf{i}\in \mathrm{Cat}_\pi$, we see that $i_1=i_3$, and $i_2,i_4$ are independent indices. Now, $G_{\tilde{\pi}}$ is a graph on 3 vertices, which are labelled as $\{\{1,3\}\}$, $\{2\}$ and $\{4\}$, and so its corresponding tuple $\tilde{\mathbf{i}}$ is exactly the same as $\mathbf{i}$.

    We then have, from \citet[Eq. 5.2.12]{Chatterjee:Hazra}, that 
    \begin{align}
        \prod_{j=1}^{M(k)}\sqrt{r_{i_ji_{j+1}}^m} = \prod_{e\in E_{\tilde{\pi}}} r_e^{t_e/2},
    \end{align}
    where $E_{\tilde{\pi}}$ is the edge set of $G_{\tilde{\pi}}$ and $t_e$ is the number of times an edge $e$ is traversed in the closed walk on $G_{\tilde{\pi}}$. Also observe
    $$
    \Phi(\tilde{\pi}) = \E\left[  \prod_{e\in E_{\tilde{\pi}}} G_e^{t_e}  \right].$$
    Consequently, we must have that $t_e$ to be even for all $e$, since the Gaussian terms are independent and mean 0. We claim that $t_e=2$ for all $e\in E_{\tilde{\pi}}$. Indeed, if for all $e$, $t_e\geq 2$ with at least one $e'$ such that $t_{e'}>2$, then, $\sum_{e\in E_{\tilde{\pi}}}t_e >2|E_{\tilde{\pi}}|$. Since $G_{\tilde{\pi}}$ is connected, $|E_{\tilde{\pi}}| \geq |V_{\tilde{\pi}}| -1 = M(k)/2$, where $V_{\tilde{\pi}}$ is the vertex set. Thus, $\sum_{e \in E_{\tilde{\pi}}} t_e> M(k)$. But, $\sum_e t_e = M(k)$, gives a contradiction. We conclude that $t_e=2$ for all $e\in E_{\tilde{\pi}}$. 

    \noindent
    A similar contradiction arises when we assume that there exists a self-loop in $G_{\tilde{\pi}}$. Thus $G_{\tilde{\pi}}$ is a tree on $\frac{M(k)}{2} + 1$ vertices with each edge traversed twice in the closed walk. As a consequence, every Gaussian term in $\Phi(\tilde{\pi})$ appears exactly twice, and so, $\Phi(\tilde{\pi}) =1$. 

\

\noindent
    Let $b_s$ be the $s^{\text{th}}$ block of $\tilde{\pi}$ and let $\ell_s$ its representative element. Define $$
    \gamma_s := \#\left\{    1\leq j\le k: 1+M(j)\in b_s       \right\},$$
    and
    $$
    \{s_1,s_2,\ldots,s_{\gamma_s}\} = \left\{   1\le j\le k: 1+M(j)\in b_s\right\}.
    $$
    
    \noindent We then have 
    \begin{align}
      \prod_{j=1}^k\left(  \frac{1}{c_N}\sum_{t=1}^N r_{i_{1+M(j)t}}     \right)^{\frac{n_j}{2}} = \prod_{s=1}^{\frac{M(k)}{2}+1}\left( \frac{1}{c_N}\sum_{t=1}^N r_{\ell_st}    \right)^{\sum_{j=1}^{\gamma_s}n_{s_j}/2}.  
    \end{align}
    Note that $$
    \sum_{j=1}^{\gamma_s}n_{s_j} = \sum_{j\in[k]:1+M(j)\in b_s} n_j.
    $$
    Let us define $\tilde{n}_s:=\sum_{j=1}^{\gamma_s}n_{s_j}/2$. Then, \begin{align}\label{eq:tilde n sum}
    \sum_{s:b_s\in\tilde{\pi}} \tilde{n}_s =\frac{N(k)}{2}.
    \end{align}
    Using \citet[Eq. 5.2.16]{Chatterjee:Hazra}, we obtain
    \begin{align}\label{eq:Laplacian variance term expansion}
        &\frac{1}{N c_N^{\frac{M(k)}{2}}}\sum_{\mathbf{i}\in \mathrm{Cat}_\pi} \prod_{j=1}^{M(k)}\sqrt{r_{i_ji_{j+1}}}\prod_{j=1}^k\left(\frac{1}{c_N}\sum_{t=1}^N r_{i_{1+M(j)}t}\right)^{\frac{n_j}{2}} \nonumber\\
        &=\frac{1}{Nc_N^{\frac{M(k)+N(k)}{2}}}\sum_{\substack{\ell_1\neq \ldots \neq \ell_{M(k)/2 +1},\\ p_{(s,1)},\ldots,p_{(s,\tilde{n}_s)}:s\in\left[ \frac{M(k)}{2}+1\right]}} \prod_{(u,v)\in E_{\tilde{\pi}}} r_{\ell_u\ell_v}\prod_{s=1}^{\frac{M(k)}{2}+1}\prod_{t=1}^{\tilde{n}_s}r_{\ell_sp_{(s,t)}},
    \end{align} 
    where for any two blocks $b_{s_1}$ and $b_{s_2}$, $\{p(s_1,1),p(s_1,2),\ldots\}$ and $\{p(s_2,1), p(s_2,2),\ldots\}$ are non-intersecting sets of indices $\{p_1,p_2,\ldots,p_{\tilde{n}_{s_1}}\}$ and $\{p_1',p_2',\ldots,p_{\tilde{n}_{s_2}}'\}$. Note that for $(u,v)\in E_{\tilde{\pi}}$, $r_{\ell_u\ell_v} = r_{uv}$ as before, but we rewrite in terms of representative elements to indicate common factors with the terms $r_{\ell_sp_{(s,t)}}$. Taking expectation in \eqref{eq:Laplacian variance term expansion} gives us 
    \begin{align}\label{eq:Laplacian varterms expectation}
        &\E[\eqref{eq:Laplacian variance term expansion}] \nonumber\\
        &= \frac{1}{Nc_N^{\frac{M(k)+N(k)}{2}}}\sum_{\ell_1\neq\ldots\neq \ell_{\frac{M(k)}{2}+1}}\E\left[ \prod_{(u,v)\in E_{\tilde{\pi}}}\frac{W_{\ell_u}^mW_{\ell_v}^m}{\|\ell_u -\ell_v\|^{\alpha}}\sum_{\substack{p_{(s,1)},\ldots,p_{(s,\tilde{n}_s)}:\\
        s\in\left[\frac{M(k)}{2}+1\right]}} \prod_{s=1}^{\frac{M(k)}{2}+1}\prod_{t=1}^{\tilde{n}_s} \frac{W_{\ell_s}^m W_{p_{(s,t)}}^m}{\|\ell_s - p_{(s,t)}\|^{\alpha}}       \right] \, .
    \end{align}
    The vertex set $V_{\tilde{\pi}}$ of the graph $G_{\tilde{\pi}}$ yields $M(k)/2 + 1$ distinct indices, due to the tree structure. Using Fact \ref{fact:Laplacian-claim410}, we see that the factor of $$
    \sum_{\ell_1,\ldots,\ell_{\frac{M(k)}{2}+1}}\prod_{(u,v)\in E_{\tilde{\pi}}} \frac{1}{\|\ell_u - \ell_v\|^{\alpha}}
    $$ is of the order of $\bigO_N\left(c_N^{\frac{M(k)}{2}}\right)$ since the weights are uniformly bounded in the range $[1,m]$. For the second summand in \eqref{eq:Laplacian varterms expectation}, the index $\ell_s$ already appears in the graph $G_{\tilde{\pi}}$, and for any $s$, we have $\tilde{n}_s$ many distinct indices from the sequence $\{p_{s,t}\}$, and summing over all $s$ yields $N(k)/2$ many distinct indices due to \eqref{eq:tilde n sum}. The second summation is therefore of the order of $\bigO_N\left(c_N^{\frac{N(k)}{2}}\right)$. 
    
      We claim that as $N\to\infty$, \eqref{eq:Laplacian varterms expectation} converges to the limit 
    $$
    \E\left[\prod_{(u,v)\in E_{\tilde{\pi}}} W_{\ell_u}^m W_{\ell_v}^m \prod_{s=1}^{\frac{M(k)}{2}+1}\prod_{t=1}^{\tilde{n}_s} W_{\ell_s}^m W_{p_{(s,t)}} \right]\, . 
    $$
    First, note that the weights are bounded, and so, \eqref{eq:Laplacian varterms expectation} is bounded above and below. Next, we note that with the scaling of $Nc_N^{M(k)/2}$, we have $$
    \lim_{N\to\infty}\frac{1}{Nc_N^{M(k)/2}}\sum_{\ell_1 \neq \ldots \neq \ell_{\frac{M(k)}{2}+1}} \E\left[    \prod_{(u,v)\in E_{\tilde{\pi}}} \frac{W^m_{\ell_u}W^m_{\ell_v}}{\|\ell_u - \ell_v\|^{\alpha}} \right] = \E\left[   \prod_{(u,v)\in E_{\tilde{\pi}}} W^m_{\ell_u}W^m_{\ell_v} \right]\, ,
    $$ 
    which is moments of the adjacency matrix of the model as in \cite{CHMS2025}. Thus, combinatorially, the first summand in \eqref{eq:Laplacian varterms expectation} corresponds with the graph $G_{\tilde{\pi}}$, as defined in Definition \ref{chap5:eq:def_c}. Now, consider a modification of the graph as follows: For each vertex $s$ in $G_{\tilde{\pi}}$, attach $\tilde{n}_s$ many independent leaves, and call the new graph $\tilde{G}_{\tilde{\pi}}$. We refer to \cite{Chatterjee:Hazra} for a detailed description, and Figure \ref{pap3fig:Laplacian moment graph} for a visual representation. 

    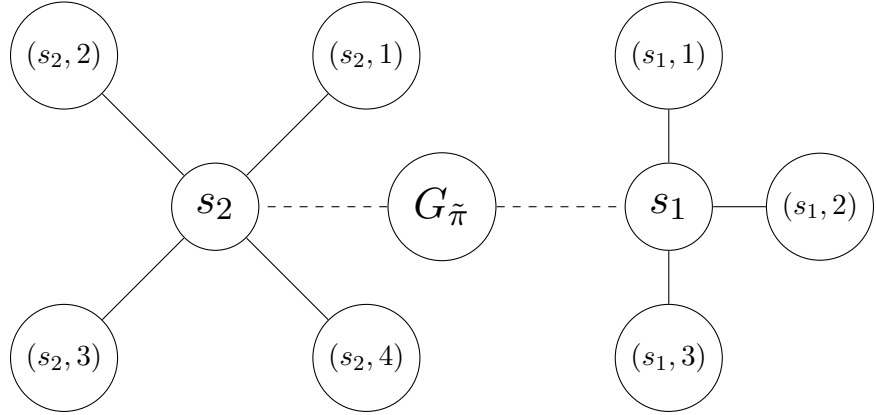
\begin{figure}[h]
    \centering
    \begin{center}
    \begin{tikzpicture}[scale=2] 
    \node[shape=circle,draw=black,scale=1.5] (s1) at (2.5,4) {$s_1$};
    \node[shape=circle,draw=black,scale=1.5] (G) at (1,4) {$G_{\tilde{\pi}}$};
    \node[shape=circle,draw=black,scale=1] (s11) at (2.5,5) {$(s_1,1)$};
    \node[shape=circle,draw=black,scale=1] (s12) at (3.5,4) {$(s_1,2)$};
    \node[shape=circle,draw=black,scale=1] (s13) at (2.5,3) {$(s_1,3)$};
    \node[shape=circle,draw=black,scale=1.5] (s2) at (-0.5,4) {$s_2$};
    \node[shape=circle,draw=black,scale=1] (s21) at (0.5,5) {$(s_2,1)$};
    \node[shape=circle,draw=black,scale=1] (s22) at (-1.5,5) {$(s_2,2)$};
    \node[shape=circle,draw=black,scale=1] (s23) at (-1.5,3) {$(s_2,3)$};
    \node[shape=circle,draw=black,scale=1] (s24) at (0.5,3) {$(s_2,4)$};
    \path [dashed] (G) edge node[left] {} (s1);
    \path [dashed] (G) edge node[left] {} (s2);
    \path [-] (s1) edge node[left] {} (s11);
    \path [-] (s1) edge node[left] {} (s12);
    \path [-] (s1) edge node[left] {} (s13);
    \path [-] (s2) edge node[left] {} (s21);
    \path [-] (s2) edge node[left] {} (s22);
    \path [-] (s2) edge node[left] {} (s23);
    \path [-] (s2) edge node[left] {} (s24);
\end{tikzpicture}
 \end{center}
    \caption{Modifying the graph $G_{\tilde{\pi}}$ to construct $\tilde{G}_{\tilde{\pi}}$. Here, we pick two vertices $s_1,s_2\in V_{\tilde{\pi}}$, with $\tilde{n}_{s_1}=3, \tilde{n}_{s_2}=4$.}
    \label{pap3fig:Laplacian moment graph}
\end{figure}
    
    The second summand over the sequence $\{p_{s,t}\}$ for each $s$ corresponds to the added leaves, since the only common index with the original graph is the index $\ell_s$ for each $s$. Keeping the index $\ell_s$ fixed (since it is summed out in the first summand involving the indices $\ell_1\neq \ldots \neq \ell_{\frac{M(k)}{2}+1}$) , we see that with the scaling $c_N^{N(k)/2}$ we have $$
    \lim_{N\to\infty}\frac{1}{c_N^{N(k)/2}} \E\left[\left. \sum_{p_{(s,1)},\ldots,p_{(s,\tilde{n}_s})} \prod_{s=1}^{\frac{M(k)}{2}+1}\prod_{t=1}^{\tilde{n}_s} \frac{W^m_{\ell_s}W^m_{p_{(s,t)}}}{\|\ell_s - p_{(s,t)}\|^{\alpha}}\right| W^m_{\ell_s} \right] = \E\left[ \left. \prod_{s=1}^{\frac{M(k)}{2}+1}\prod_{t=1}^{\tilde{n}_s} W^m_{\ell_s}W^m_{p_{(s,t)}}\right| W^m_{\ell_s}\right]\, . 
    $$
    Due to the compact support of the weights, it is now easy to conclude that 
    \begin{align}\label{eq:Laplacian moments homomorphism density}
         \lim_{N\to\infty}\eqref{eq:Laplacian varterms expectation} = \E\left[\prod_{(u,v)\in E_{\tilde{\pi}}} W_{\ell_u}^m W_{\ell_v}^m \prod_{s=1}^{\frac{M(k)}{2}+1}\prod_{t=1}^{\tilde{n}_s} W_{\ell_s}^m W_{p_{(s,t)}} \right] =: t(\tilde{G}_{\tilde{\pi}},W^m)
    \end{align}
    where $W^m=(W_1^m,W_2^m,\ldots)$ and $\tilde{G}_{\tilde{\pi}}$ is the modified graph as described above and illustrated in Figure \ref{pap3fig:Laplacian moment graph}.

    We can therefore conclude that for all even $k$,
    \begin{align}\label{eq:Laplacian expected moments}
        \lim_{N\to\infty}\frac{1}{N}\E\left[\tr(\bfD_{N,g,c}^k)\right] = \sum_{\substack{m_1,\ldots,m_k,\\n_1,\ldots,n_k}} \sum_{\pi\in NC_2(M(k))} \mathcal{E}(\tilde{\pi})t(\tilde{G}_{\tilde{\pi}},W^m) \, .
    \end{align}
    
    \
    
    Now, consider the case when $k$ is odd. Due to \eqref{eq:Laplacian n_j condition}, we have that $M(k)$ must be odd. Thus, $\pi$ cannot be a pair partition, and in particular, $\pi\not\in NC(M(k))$. Consider the term $\Phi(\tilde{\pi})$ in \eqref{eq:Laplacian moment phi term}, and notice that by Wick's formula, this term is identically 0 if $M(k)$ is odd. Since the other expectations in \eqref{eq:Laplacian moment expectation expansion} are of order $\bigO_N(1)$, we conclude that the odd moments are 0 in expectation. 

    \

   \paragraph{Step 2.} We now wish to show concentration of the moments. Define 
    $$
    P(\mathbf{i}) := \E\left[\prod_{j=1}^{M(k)} G_{i_j \wedge i_{j+1}, i_j\vee i_{j+1}}\prod_{j=1}^{M(k)}\frac{\sqrt{r_{i_ji_{j+1}}}}{\sqrt{c_N}} \prod_{j=1}^k \left( \frac{1}{c_N}\sum_{t=1}^N r_{i_{1+M(j)}t} \right)^{\frac{n_j}{2}}\prod_{j=1}^k Z_{i_{1+M(j)}}^{n_j}\right] \, ,
    $$
    and 
    \begin{align*}
    P(\mathbf{i},\mathbf{i}') &:= \E\left[\prod_{j=1}^{M(k)} G_{i_j \wedge i_{j+1}, i_j\vee i_{j+1}}\prod_{j=1}^{M(k)}\frac{\sqrt{r_{i_ji_{j+1}}}}{\sqrt{c_N}} \prod_{j=1}^k \left( \frac{1}{c_N}\sum_{t=1}^N r_{i_{1+M(j)}t} \right)^{\frac{n_j}{2}}\prod_{j=1}^k Z_{i_{1+M(j)}}^{n_j}\right. \\
    \times& \left.\prod_{j=1}^{M(k)} G_{i'_j \wedge i'_{j+1}, i'_j\vee i'_{j+1}}\prod_{j=1}^{M(k)}\frac{\sqrt{r_{i'_ji'_{j+1}}}}{\sqrt{c_N}} \prod_{j=1}^k \left( \frac{1}{c_N}\sum_{t=1}^N r_{i'_{1+M(j)}t} \right)^{\frac{n_j}{2}}\prod_{j=1}^k Z_{i'_{1+M(j)}}^{n_j}\right] \, .
    \end{align*}
    Then, 
    \begin{align}\label{eq:Laplacian moment variance step 1}
    \Var\left(\int_{\mathbb{R}}x^{k}\ESD(\bfD_{N,g,c})(\dd{x})\right) = \frac{1}{N^2}\sum_{\substack{m_1,\ldots,m_k,\\n_1,\ldots,n_k}}\sum_{\mathbf{i},\mathbf{i}':[M(k)]\to [N]} \left[P(\mathbf{i},\mathbf{i}') - P(\mathbf{i})P(\mathbf{i}')\right]\, ,
    \end{align}
    and  we would like to show $\eqref{eq:Laplacian moment variance step 1}\to 0$. 
    If $\mathbf{i}$ and $\mathbf{i}'$ have no common indices, then $P(\mathbf{i},\mathbf{i}')=P(\mathbf{i})P(\mathbf{i}')$ by independence. If there is \emph{exactly} one common index, say $i_1=i'_1$, then by independence of Gaussian terms, the factors $\E[G_{i_1,i_2}]$ and $\E[G_{i_1,i_2'}]$ would pull out, causing \eqref{eq:Laplacian moment variance step 1} to be identically 0. Thus, we have \emph{at least} one matching of the form $(i_1,i_2)=(i_1',i_2')$. 

    Let us begin by taking $k$ to be even. Consider exactly one matching, which we take to be $(i_1,i_2)=(i_1',i_2')$ without loss of generality. Let $\pi,\pi'$ be partitions of $\{1,2,\ldots,M(k)\},\{1',2',\ldots,M(k)'\}$ respectively. Let $\mathop{\sum\nolimits^{(1)}}$ denote the sum over index sets $\mathbf{i},\mathbf{i}'$ with exactly one matching. Then, we have by an extension of the previous argument
    \begin{align}\label{eq:Laplacian moment variance bound}
        \frac{1}{N^2}\mathop{\sum\nolimits^{(1)}}_{\mathbf{i},\mathbf{i}':[M(k)]\to[N]}P(\mathbf{i},\mathbf{i}') \leq \frac{1}{N^2c_N^{M(k)}}\sum_{\pi,\pi'}\Phi(\tilde{\pi})\mathcal{E}(\tilde{\pi})\Phi(\tilde{\pi}')\mathcal{E}(\tilde{\pi}')\sum_{\mathbf{i},\mathbf{i}'}\E\left[ r_{i_1i_2}\prod_{j=2}^{M(k)}\sqrt{r_{i_ji_{j+1}}}\prod_{j=2}^{M(k)}\sqrt{r_{i_j'i_{j+1}'}} \right]\, .
    \end{align}
    Expanding the expression for $r_{ij}$ and using the fact that $W_i^m\leq m$ gives us that
    \begin{align}\label{eq:Laplacian moment variance second bound}
        \eqref{eq:Laplacian moment variance bound} \leq \frac{m^{2M(k)}}{N^2c_N^{M(k)}}\sum_{\pi}\Phi(\tilde{\pi})\mathcal{E}(\tilde{\pi})\Phi(\tilde{\pi}')\mathcal{E}(\tilde{\pi}')\sum_{\mathbf{i,i}'}\frac{1}{\|i_1-i_2\|^{\alpha}}\prod_{j=1}^{M(k)}\frac{1}{\|i_j - i_{j+1}\|^{\alpha/2}}\frac{1}{\|i_j'-i_{j+1}'\|^{\alpha/2}}\, .
    \end{align}
We are now precisely in the setting of \cite{CHMS2025}, and in particular, following the ideas from \citet[p24]{CHMS2025} and using Fact \ref{fact:Laplacian-claim410}, we obtain that the right-hand side of \eqref{eq:Laplacian moment variance second bound} is of order $\bigO_N(c_N^{-1})$. For $t$ matchings in $\mathbf{i,i}'$, the order is $\bigO(c_N^{-t})$, giving us that \eqref{eq:Laplacian variance term expansion} is of order $\bigO(c_N^{-1})$ when $k$ is even. 

The argument for the case where $k$ is odd is similar. Since the optimal order is achieved when we take $\mathbf{i}\setminus\{i_1,i_2\}\in\mathrm{Cat}_{\pi}$ and $\mathbf{i}'\setminus\{i_1',i_2'\}\in\mathrm{Cat}_{\pi'}$, with $\pi,\pi'\in NC_2(M(k))$, one cannot construct these partitions with $k$ being odd with the restriction from \eqref{eq:Laplacian n_j condition} imposing that $M(k)$ must be odd. Consequently, we have convergence in $\prob$--probability of the moments of $\ESD(\bfD_{N,g,c})$. Thus, we conclude that $$
\lim_{N\to\infty}\tr(\bfD_{N,g,c}^k) = M_k \quad \text{in $\prob$--probability}\, ,
$$
where \begin{align}\label{eq:Laplacian moment expression}
    M_k = \begin{cases}
        \sum_{\mathcal{M}(k)}\sum_{\pi\in NC_2(M(k))} t(\tilde{G}_{\tilde{\pi}},W^m)\mathcal{E}(\tilde{\pi}) \,, &\text{$k$ even},\\
        0 \,, &\text{$k$ odd}\,,
    \end{cases}
\end{align}
where $\mathcal{M}(k)$ is the multiset of all numbers $(m_1,\ldots,m_k,n_1,\ldots,n_k)$ that appear in the expansion $(a+b)^k$ for two non-commutative variables $a$ and $b$. 

\

\paragraph{Step 3.} We are now left to show that these moments uniquely determine a limiting measure. This follows from \citet[Section 5.2.2]{Chatterjee:Hazra}, but we show the bounds for the sake of completeness. 

First, from \citet[Section 5.2.2]{Chatterjee:Hazra}, we have that $\mathcal{E}(\tilde{\pi})\le 2^k k!$. Next, observe from \eqref{eq:Laplacian moments homomorphism density} that $|t(\tilde{G}_{\tilde{\pi}},W^m) | \le (m^2)^{\frac{k}{2}} = m^2$, since $W_i\le m$ for all $i$ and $\tilde{G}_{\tilde{\pi}}$ is a graph on $\frac{k}{2}+1$ vertices with $\frac{k}{2}$ edges. Lastly, $|NC_2(M(k)|\le |NC_2(k)| = C_k$, where $C_k$ is the $k^{\text{th}}$ Catalan number, and moreover, $|\mathcal{M}(k)| \le 2^k$. Combining these, we have 
$$
\beta_k := |M_k| \leq 2^k . C_k. m^k.2^k k! = (4m)^kC_kk!\,.
$$
Using Sterling's approximation, we have 
$$
\frac{1}{k}\beta_k^{\frac{1}{k}} \le \frac{4m}{(k+1)^{\frac{1}{k}}}.\frac{4\e^{-\left(  1 + \frac{1}{k}    \right)}}{\pi^{\frac{1}{k}}}\, ,
$$
where $\pi$ here is now the usual constant, and subsequently, we have \begin{align}\label{eq:Laplacian carleman}
\limsup_{k\to\infty}\frac{1}{2k}\beta_{2k}^{\frac{1}{2k}} < \infty \, .
\end{align}
 Equation \eqref{eq:Laplacian carleman} is a well known criteria to show that the moments uniquely determine the limiting measure (see \citet[Theorem 1]{Lin-Carleman}).
This completes the argument. 
\end{proof} 

\section{Identification of the limit: Proof of Theorem \ref{theorem: Laplacian measure identification}}
\label{sec:final proof Laplacian}
\subsection{The case $\alpha=0$}\label{subsec:Laplacian alpha=0} In Section \ref{sec:moment method Laplacian}, we show the existence of a unique limiting measure $\nu_\tau$ such that $$
\lim_{N\to\infty}\ESD(\bfD_N^{\circ})=\nu_\tau\qquad \text{in $\prob$--probability}\,.
$$
We have also shown that $\nu_\tau$ is the limiting measure for the $\ESD$ of the Laplacian matrix $\widehat{\bfD}_{N,g}$. 
In particular, through the proof of Proposition \ref{main:prop}, we show that the limit $\nu_{\tau,m}$ is independent of the choice of $\alpha$, and consequently, $\nu_\tau$ is $\alpha$--independent. We then use the idea of substituting $\alpha=0$ from \citet[Section 6]{CHMS2025} in the matrix $\widehat{\bfD}_{N,g}$, to obtain the Laplacian matrix $\bfD_{N,g}^\circ$, which corresponds to the adjacency matrix $\A_{N,g}^\circ$ with entries given by 
\begin{align*}
    \A_{N,g}^{\circ}(i,j) = \begin{cases}
        \frac{\sqrt{W_iW_j}}{\sqrt{N}}G_{i\wedge j,i\vee j}, \, &i\neq j\\
        0, \, &i=j.
    \end{cases}
\end{align*}
Then, $\lim_{N\to\infty}\ESD(\bfD_{N,g}^\circ) = \nu_\tau$ in $\prob$--probability. Recall that for all $1\le i\le N$, $W_i^m := W_i\one_{W_i\leq m}$ for any $m\geq 1$. We can now apply Lemmas \ref{lemma:truncation} and \ref{lemma:decoupling} to contruct a matrix $\bfD_{N,g,c}^\circ = \A_{N,g,m}^\circ + Y_N^\circ$ such that $$
\limsup_{m\to\infty}\lim_{N\to\infty}\prob\left(  d_L(\ESD(\bfD_{N,g}^\circ), \ESD(\bfD_{N,g,c}^{\circ}) )>\delta    \right) =0\, ,
$$ where 
\begin{align*}
    \A_{N,g,m}^{\circ}(i,j) = \begin{cases}
        \frac{\sqrt{W_i^m W_j^m}}{\sqrt{N}}G_{i\wedge j,i\vee j}, \, &i\neq j\\
        0, \, &i=j,
    \end{cases}
\end{align*}
and $Y_N^\circ$ is a diagonal matrix with entries $$
Y_N^\circ(i,i) = Z_i\sqrt{\frac{\sum_{k\neq i}W_i^m W_k^m}{N}}\, .
$$
By Proposition \ref{main:prop}, we have that $\lim_{N\to\infty}\ESD(\bfD_{N,g,c}^\circ) = \nu_{\tau,m}$ in $\prob$--probability. Thus, we begin by identifying $\nu_{\tau,m}$. To that end, consider the matrix $\widehat{\bfD}_{N,g,c}^\circ := \A_{N,g,c} + \widehat{Y}_N^\circ$, with $\A_{N,g,c}$ as before, and $\widehat{Y}_N^\circ$ a diagonal matrix with entries 
$$
\widehat{Y}_N^\circ(i,i) = Z_i\sqrt{W_1^m}\sqrt{\mathbf{E}[W_1^m]}\,. 
$$
We now have the following lemma. 
\begin{lemma}\label{lemma:LLN argument}
    Let $\bfD_{N,g,c}^\circ$ and $\widehat{\bfD}_{N,g,c}^\circ$ be as defined above. Then, 
    $$
    \lim_{N\to\infty}\prob\left( d_L(\ESD(\bfD_{N,g,c}^\circ), \ESD(\widehat{\bfD}_{N,g,c}^\circ))>\delta \right) =0\, .
    $$
\end{lemma}
\begin{proof}
    We apply Proposition \ref{prop:hoffman-wielandt} to obtain 
    \begin{align}\label{eq:LLN argument}
        \E\left[ d_L(\ESD(\bfD_{N,g,c}^\circ), \ESD(\widehat{\bfD}_{N,g,c}^\circ))  \right] &\leq \frac{1}{N}\E\sum_{i=1}^N \left( Y_N^\circ(i,i) - \widehat{Y}_N^\circ(i,i) \right)^2 \nonumber\\
        &\leq \frac{1}{N}\E[Z_1^2]\ep[W_1^m]\sum_{i=1}^N\ep\left[  \left(  \frac{\sqrt{\sum_{k\neq i}W_k^m}}{\sqrt{N}} - \sqrt{\ep[W_1^m] }  \right)^2  \right] \nonumber \\
        &\leq \frac{m}{N}\sum_{i=1}^N \ep\left[\left|   \frac{\sum_{k=1}^N W_k^m}{N} - \ep[W_1^m]     \right|\right] \, . 
    \end{align}
    We have that $(W_i^m)_{i\in\Ver_N}$ is a bounded sequence of ~i.~i.~d. random variables, and in particular have finite variance. By the strong law of large numbers, we have that $$
    \lim_{N\to\infty}\frac{\sum_{k=1}^N W_k^m}{N} = \mathbf{E}[W_1^m]\quad \text{$\prob$--almost surely}\,.
    $$ 
    
However, by the boundedness of the weights, we have that $N^{-1}\sum_{i=1}^N W_i^m$ is uniformly bounded by $m$, which is integrable (with respect to $\E$). By the dominated convergence theorem, we have convergence in $L^1$, and consequently, \eqref{eq:LLN argument} goes to 0 as $N\to\infty$. We conclude with Markov's inequality. 
\end{proof}

We can now conclude that $\nu_{\tau,m}$ is the limiting measure of the $\ESD$ of the matrix $\widehat{\bfD}_{N,g,c}^\circ$. 

\subsection{Identification of $\nu_{\tau,m}
$} 
We have that 
$$
\lim_{N\to\infty}\ESD(\widehat{\bfD}_{N,g,c}^\circ) = \nu_{\tau,m} \quad \text{in $\prob$--probability}\,.$$ Notice that $\widehat{\bfD}_{N,g,c}^\circ$ can be written as \begin{align*}
    \widehat{\bfD}_{N,g,c}^\circ &= \A_{N,g,m}^\circ + \widehat{Y}_N^\circ \\
    &= \mathrm{W}_m^{1/2}\left(\frac{1}{\sqrt{N}}\mathrm{G}\right)\mathrm{W}_m^{1/2} + \sqrt{\mathbf{E}[W_1^m]}\mathrm{W}_m^{1/4}\mathrm{Z}\mathrm{W}_m^{1/4} \, ,
\end{align*}
where $\mathrm{W}_m = \diag(W_1^m,\ldots,W_N^m)$, $\mathrm{G}$ is a standard Wigner matrix with ~i.~i.~d $\gauss(0,1)$ entries above the diagonal and 0 on the diagonal, and $\mathrm{Z}$ is a diagonal matrix with ~i.~i.~d. $\gauss(0,1)$ entries. 

First we need to show that
\begin{align*}
 & \lim _{N \rightarrow \infty} \ESD\left(\mathrm{W}_m^{1/2}\left(\frac{1}{\sqrt{N}}\mathrm{G}\right)\mathrm{W}_m^{1/2} + \sqrt{\mathbf{E}[W_1^m]}\mathrm{W}_m^{1/4}\left(\frac{1}{\sqrt{N}}\mathrm{Z}\right)\mathrm{W}_m^{1/4} \right) \\ &=\mathcal{L}\left(T_{W_m}^{1/2} T_s T_{W_m}^{1 / 2}+\sqrt{\E[W_m]} T_{W_m}^{1/4} T_g T_{W_m}^{1/4}\right) \text { weakly in probability }.
\end{align*}

This easily follows by retracing the arguments in the proof of \cite[Theorem 1.3]{CHHS} and using the Lemma \ref{lemma:free prod} presented in the appendix. This shows that
$$\nu_{\tau, m}=\mathcal{L}\left(T_{W_m}^{1/2} T_s T_{W_m}^{1 / 2}+\sqrt{\E[W_m]} T_{W_m}^{1/4} T_g T_{W_m}^{1/4}\right).$$

\subsection{Identification of $\nu_\tau$} We now conclude with the proof of Theorem \ref{theorem: Laplacian measure identification}. 
\begin{proof}[Proof of Theorem \ref{theorem: Laplacian measure identification}]
We work in a tracial $W^*$-probability space $(\mathcal A,\varphi)$. Let
$T_g$ and $T_s$ denote, respectively, a standard Gaussian and a standard semicircular self-adjoint element.
For each $m\in\mathbb N$, let $T_{W_m}$ (resp. $T_W$) be self-adjoint affiliated operators with
distributions $\mu_{W_m}$ (resp. $\mu_W$). We assume $T_s$ is free from the abelian von Neumann
algebras $W^*(T_{W_m},T_g)$ and $W^*(T_W,T_g)$. By the freeness result stated in \citet[Proposition 4.1]{bercovici1993free}, this implies that for any bounded Borel $u,v$ the elements $u(T_{W_m})v(T_g)$ and $u(T_W)v(T_g)$ are free from $T_s$.

Fix the constant $c:=\sqrt{\E[W]}$. If one writes $c_m:=\sqrt{\E[W_m]}$, then $c_m\to c$. The replacement
$c_m\mapsto c$ is justified at the end. For $R>0$ define bounded truncations
$$u_R(x):=(x \wedge R)^{1/2}, \quad v_R(x):=(x \wedge R)^{1/4}.$$

Define affiliated (self-adjoint) operators 
\begin{equation}\label{eq:defQR} 
Q_R^{(m)} := u_R(T_{W_m})\,T_g\,u_R(T_{W_m}) + c\,v_R(T_{W_m})\,T_s\,v_R(T_{W_m}),
\end{equation}
and
\begin{equation}
Q_R := u_R(T_W)\,T_g\,u_R(T_W) + c\,v_R(T_W)\,T_s\,v_R(T_W),
\end{equation} 
noting that $u_R(\cdot)$ and $v_R(\cdot)$ are bounded while $T_g$ may be
unbounded, so $Q_R^{(m)},Q_R$ are in general affiliated rather than bounded.

Let $$\varepsilon_m:=d_\infty(\mu_{W_m},\mu_W)=\sup_{t\in\mathbb R}|
F_{\mu_{W_m}}(t)-F_{\mu_W}(t)|\to0.$$
Theorem 3.9 of \citet{bercovici1993free} allows us to represent $T_{W_m}$ and $T_W$ in the same space together with a projection $p_m\in W^*(T_{W_m},T_W,T_g)$ such that \begin{equation}
\varphi(p_m)\ge 1-\varepsilon_m,\qquad p_m\,T_{W_m}\,p_m = p_m\,T_W\,p_m. \label{eq:BVcompress}
\end{equation} 

Since $p_m$ lies in the abelian algebra generated by $T_{W_m},T_W,T_g$, it commutes with
$T_g$ and with bounded Borel functions of $T_{W_m},T_W$. Hence $p_m u_R(T_{W_m})=u_R(T_W)p_m$ and
$p_m v_R(T_{W_m})=v_R(T_W)p_m$, and therefore 
\begin{equation} 
p_m\,Q_R^{(m)}\,p_m\ =\
p_m\,Q_R\,p_m. \label{eq:compressQR} 
\end{equation} 

From the estimate (\citet[Theorem 3.9 (i)]{bercovici1993free})
if $A,B$ are selfadjoint and $pAp=pBp$ with $\varphi(p)\ge 1-\delta$, then $$\sup_t|F_A(t)-F_B(t)|
\le \delta$$  we obtain, for each fixed $R>0$, 
\begin{equation} d_\infty\big(\mathcal L(Q_R^{(m)}),\mathcal
L(Q_R)\big)\ \le\ 2\varepsilon_m\ \xrightarrow[m\to\infty]{}\ 0. \label{eq:fixedRconv} 
\end{equation}

We introduce the following affiliated operators

$$H_m:=T_{W_m}^{1 / 2} T_g T_{W_m}^{1 / 2}+c T_{W_m}^{1 / 4} T_s T_{W_m}^{1 / 4}, \quad H:=T_W^{1 / 2} T_g T_W^{1 / 2}+c T_W^{1 / 4} T_s T_W^{1 / 4}.$$

Define the projections $q_{m,R}:=1_{[0,R]}(T_{W_m})$ and $q_R:=1_{[0,R]}(T_W)$. On the range $\mathrm{Ran}(q_{m,R})$ we have
$u_R(T_{W_m})=T_{W_m}^{1/2}$ and $v_R(T_{W_m})=T_{W_m}^{1/4}$, and similarly with $W$ in place of
$W_m$. Consequently,
$$q_{m, R} Q_R^{(m)} q_{m, R}=q_{m, R} H_m q_{m, R}, \quad q_R Q_R q_R=q_R H q_R.$$

Applying the same compression–distribution estimate we have
\begin{equation} \label{eq:Rtails}
d_\infty\big(\mathcal
L(Q_R^{(m)}),\mathcal L(H_m)\big)\ \le\ 2\,\varphi(1-q_{m,R})=2\,\big(1-F_{\mu_{W_m}}(R)\big),
\end{equation}
and 
\begin{equation}\label{eq:Rtails2}
d_\infty\big(\mathcal L(Q_R),\mathcal L(H)\big)\ \le\ 2\,\big(1-F_{\mu_W}(R)\big). 
\end{equation} 
Since $\mu_{W_m}\Rightarrow\mu_W$, the family ${\mu_{W_m}}$ is tight; hence $
\sup_m(1-F_{\mu_{W_m}}(R))\to0$ and $1-F_{\mu_W}(R)\to0$ as $R\to\infty$. Therefore, \begin{equation}
\mathcal L(Q_R^{(m)})\Rightarrow\mathcal L(H_m)\quad\text{and}\quad \mathcal L(Q_R)
\Rightarrow\mathcal L(H)\qquad(R\to\infty), \label{eq:Rlimit} \end{equation} with convergence uniform in
$m$ along the first convergence.

Fix $\eta>0$. Choose $R$ so large that both bounds in
\eqref{eq:Rtails} and \eqref{eq:Rtails2} are less than $\eta$ for all large $m$. Then choose $m$ large so that \eqref{eq:fixedRconv} gives
$d_\infty(\mathcal L(Q_R^{(m)}),\mathcal L(Q_R))\le\eta$. The triangle inequality yields
\begin{align*}
d_{\infty}\left(\mathcal{L}\left(H_m\right), \mathcal{L}(H)\right) &\leq d_{\infty}\left(\mathcal{L}\left(H_m\right), \mathcal{L}\left(Q_R^{(m)}\right)\right)+d_{\infty}\left(\mathcal{L}\left(Q_R^{(m)}\right), \mathcal{L}\left(Q_R\right)\right)+d_{\infty}\left(\mathcal{L}\left(Q_R\right), \mathcal{L}(H)\right)\\
&\leq 3 \eta
\end{align*}

As $\eta\downarrow0$ we conclude 
\begin{equation} \mathcal L(H_m)\ \Rightarrow\ \mathcal L(H)
\qquad(m\to\infty). \label{eq:HmtoH} 
\end{equation} 
This identifies $\nu_\tau=\mathcal L(H)$.

Finally, to replace $c_m=\sqrt{\E[W_m]}$ by $c=\sqrt{\E[W]}$, note that
\[
H_m^{(c)}-H_m^{(c_m)}=(c-c_m)\,T_{W_m}^{1/4}T_sT_{W_m}^{1/4}.
\]
Fix $M>0$ and set $r_{m,M}=1_{[0,M]}(T_{W_m})$. Then
\[
\|r_{m,M}(H_m^{(c)}-H_m^{(c_m)})r_{m,M}\,\|\le |c-c_m|\,M^{1/2}\,\|T_s\|.
\]
Therefore, for the bounded–Lipschitz distance $d_{BL}$,
\begin{equation}\label{finaldbl}
d_{BL}\left(\mathcal L(H_m^{(c)}),\mathcal L(H_m^{(c_m)})\right)
\le 2\,\varphi(1-r_{m,M}) + |c-c_m|\,M^{1/2}\,\|T_s\|.
\end{equation}

Indeed, on $\operatorname{Ran}\left(r_{m, M}\right)$,
$$
\left\|r_{m, M}\left(H_m^{(c)}-H_m^{\left(c_m\right)}\right) r_{m, M}\right\|=\left|c-c_m\right|\left\|r_{m, M} T_{W_m}^{1 / 4} T_s T_{W_m}^{1 / 4} r_{m, M}\right\| \leq\left|c-c_m\right| M^{1 / 2}\left\|T_s\right\|=: \varepsilon .
$$

For any bounded 1 -Lipschitz $f$ with $\|f\|_{\infty} \leq 1$,

$$
\begin{aligned}
\left|\varphi\left(f\left(H_m^{(c)}\right)\right)-\varphi\left(f\left(H_m^{\left(c_m\right)}\right)\right)\right| \leq & \left|\varphi\left(r_{m, M} f\left(H_m^{(c)}\right) r_{m, M}-r_{m, M} f\left(H_m^{\left(c_m\right)}\right) r_{m, M}\right)\right| \\
& +2 \varphi\left(1-r_{m, M}\right) \\
\leq & \operatorname{Lip}(f) \varepsilon+2 \varphi\left(1-r_{m, M}\right) \leq \varepsilon+2 \varphi\left(1-r_{m, M}\right)
\end{aligned}
$$

where we used $\|r f r-r g r\| \leq\|f-g\| \leq \operatorname{Lip}(f)\left\|r\left(H_m^{(c)}-H_m^{\left(c_m\right)}\right) r\right\|$ on the range of $r=r_{m, M}$. Taking sup over such $f$ gives \eqref{finaldbl}.
Since $c_m\to c$ and $\sup_m\varphi(1-r_{m,M})\to 0$ as $M\to\infty$ by tightness of $(\mu_{W_m})$, we obtain
$$d_{BL}\big(\mathcal L(H_m^{(c)}),\mathcal L(H_m^{(c_m)})\big)\to 0$$ as $m\to\infty$.
Consequently, the limit law is the same whether we use $c_m$ or the constant $c$.

\end{proof}
\section*{Acknowledgements} 
\begin{wrapfigure}{l}{0.07\textwidth} 
  \vspace{-\intextsep}   \includegraphics[width=0.07\textwidth]{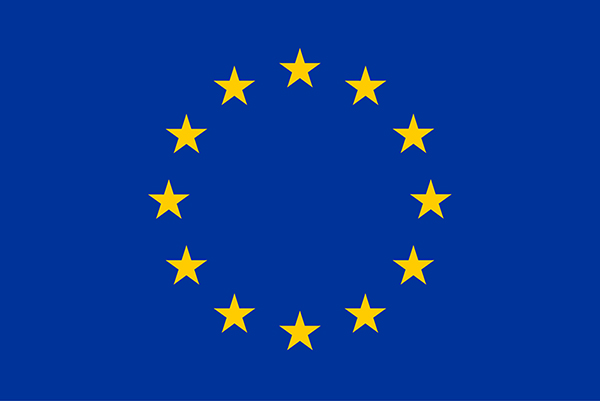}
\end{wrapfigure}
\noindent  The work of R.S.H.~and N.M.~is supported in part by the Netherlands Organisation for Scientific Research (NWO) through the Gravitation NETWORKS grant 024.002.003. The work of N.M. is further supported by the European Union’s Horizon 2020 research and innovation programme under the Marie Skłodowska-Curie grant agreement no. 945045.

\section{Appendix}\label{appendix:prel_lemmas}

\subsection{Proof of Lemma \ref{lemma:Laplacian Gaussianisation} }

\begin{proof}[Proof of Lemma \ref{lemma:Laplacian Gaussianisation}]
    Following the proof of \cite{CHMS2025} for the Laplacian, we define, conditional on the weights $(W_i)_{i\in \Ver_N}$,
    a sequence of independent random variables. Let $\Xvec_b = (X_{ij}^b)_{1\le i<j\le N}$ be a vector with  $X_{ij}^{b} \sim \Ber(p_{ij}) - \E[p_{ij}]$. Similarly, take another vector $\Xvec_g = (X_{ij}^g)_{1\le i<j\le N}$ with $X_{ij}^g \sim \gauss\left(\mu_{ij}\, , p_{ij}(1-p_{ij})\right)$. 

    Let $n=N(N-1)/2$ and $\xvec = (x_{ij})_{1\leq i<j\leq N} \in \Rr^{n}$. Define $\mathrm{R}(\xvec)$ to be the matrix-valued differentiable function given by 
    \[
    \mathrm{R}(\xvec) \coloneqq (\M_N(\xvec)-z\Id_N)^{-1}, 
    \]
    where $\M_N(\cdot)$ is the matrix-valued differentiable function that maps a vector in $\Rr^{n}$ to the space of $N\times N$ Hermitian matrices, given by 
    \[
    \M_N(\xvec)_{ij} = \begin{cases} 
        c_N^{-1/2} x_{ij} &\text{if $i< j$},\\
        c_N^{-1/2} x_{ji} &\text{if $i>j$},\\
        -c_N^{-1/2}\sum_{k\neq i}x_{ik} & \text{if $i=j$}.
    \end{cases}
    \] 
    Then, we see that $\bfD_N^o = \M_N(\Xvec_b)$ and $\bar{\bfD}_N = \M_N(\Xvec_g)$. Note that $$E^W[X_{ij}^b] = E^W[X_{ij}^g] = \mu_{ij},$$ and $$E^W[(X_{ij}^b)^2] = E^W[(X_{ij}^g)^2] = p_{ij} + \E[p_{ij}]^2 - 2p_{ij}\E[p_{ij}].$$  Consequently, using \cite[Theorem 1.1]{chatterjee2005simple} we have that 
    \begin{align}
        &\left| \E[ h(\Re H_z(\bar{\bfD}_N))]- \E[ h(\Re H_z(\bfD_{N}^o))]\right|\nonumber\\
        &=\left| \mathbf{E}\left[ E^W[h(\Re H_z(\bar{\bfD}_N))- h(\Re H_z(\bfD_{N}^o))]\right]\right|\nonumber\\
         &\leq C_1(h)\lambda_2(H)\sum_{1\le i<j\le N} \E[(X_{ij}^b)^2\one_{|X_{ij}^b|>K_N}] +  \E[(X_{ij}^g)^2\one_{|X_{ij}^g|>K_N}] \label{eq:Chatterjee<K}\\
    &+ C_2(h)\lambda_3(H)\sum_{1\le i<j\le N} \E[(X_{ij}^b)^3\one_{|X_{ij}^b|\leq K_N}] +  \E[(X_{ij}^g)^3\one_{|X_{ij}^g|\leq K_N}] 
    \end{align}

where $\lambda_2(H)\le C_2(\Im z)\frac{1}{Nc_N}$ and $\lambda_3(H)\le C_3(\Im z)\frac{1}{Nc_N^{3/2}}$.

We first deal with the terms in \eqref{eq:Chatterjee<K}. Note that since $p_{ij}\leq 1$, we have $|X_{ij}^b|\le 1$, and as a consequence, for any $K_N\ge 2$, the first term in \eqref{eq:Chatterjee<K} is zero. For the Gaussian term,  applying the Cauchy-Schwarz inequality followed by the second-moment Markov inequality yields
$$\E[(X_{ij}^g)^2\one_{|X_{ij}^g|>K_N}]\le \E[ (X_{ij}^g)^4]^{1/2} \prob(X_{ij}^g > K_N)^{1/2}\le K_N^{-1}\E[(X_{ij}^g)^4]^{1/2}\E[(X_{ij}^g)^2]^{1/2}.$$

Since $\E[(X_{ij}^g)^2]=\mathbf{E}[p_{ij} + \E[p_{ij}]^2 - 2p_{ij}\E[p_{ij}]] \le \mathbf{E}[p_{ij}]$, and similarly, $\E[(X_{ij}^g)^4]\le \mathbf{E}[p_{ij}^2]$, we have 
\begin{align*}
\lambda_2(H)\sum_{1\le i<j\le N}\E[(X_{ij}^g)^2\one_{|X_{ij}^g|>K_N}]&\le \frac{\lambda_2(H)}{K_N}\sum_{1\le i<j\le N}\frac{\mathbf{E}[W_i^2]^{1/2}\mathbf{E}[W_j^2]^{1/2}}{\|i-j\|^\alpha}\frac{\mathbf{E}[W_i]^{1/2}\mathbf{E}[W_j]^{1/2}}{\|i-j\|^{\alpha/2}}\\
&\le \frac{\lambda_2(H)}{K_N}\mathbf{E}[W_1]\mathbf{E}[W_1^2] N^{2-\frac{3\alpha}{2}}\\
&\le \frac{\tilde{c_2}\mathbf{E}[W_1]\mathbf{E}[W_1^2] N^{2-\frac{3\alpha}{2}} }{K_NN^{2-\alpha}}= \bigO_N(N^{-\alpha/2}K_N^{-1}),
\end{align*}
where the last equality follows as $\tau>3$ and $\tilde{c_2}$ is a constant depending on $\Im(z)$ only.  For the term containing the third moments, we see that 

\begin{align*}
\lambda_3(H)\sum_{1\le i<j\le N} \E[(X_{ij}^b)^3\one_{|X_{ij}^b|\leq K_N}] +  \E[(X_{ij}^g)^3\one_{|X_{ij}^g|\leq K_N}] &\le \lambda_3(H) K_N\sum_{1\le i\le j\le N}\E[(X_{ij}^b)^2]+\E[(X_{ij}^g)^2]\\
&\le \lambda_3(H) K_N 2\mathbf{E}[W_1]^2\sum_{1\le i\le j\le N}\frac{1}{\|i-j\|^\alpha}\\
&\le \frac{c_3(\Im z)}{N c_N^{3/2}} K_N \mathbf{E}[W_1]^2N c_N\le \tilde{c}_3 K_Nc_{N}^{-1/2}.
\end{align*}
Here $\tilde{c}_3$ is a constant depending on $\Im(z)$. Choosing any $2\le K_N\ll c_N^{1/2}$, both terms go to zero. This completes the proof of the Gausssinisation. 
\end{proof}

\subsection{Technical lemmas}\label{sec: technical lemmas SFP} 
In this subsection we collect some technical lemmas that are used in the proofs of our main results.
For bounding the $d_L$ distance between the ESDs of two matrices, we will need the following inequality, due to Hoffman and Wielandt.

\begin{proposition}[Hoffman-Wielandt inequality{~\citep[Corollary A.41]{Bai-silverstein}}]\label{prop:hoffman-wielandt}
Let $\mathbf{A}$ and $\mathbf{B}$ be two $N \times N$ normal matrices and let  $\ESD(\A)$ and $\ESD(\mathbf{B})$  be their ESDs, respectively. Then,
\begin{equation}\label{HWinequality}
d_L\left(\ESD(\A), \ESD(\mathbf{B})\right)^3 \leq \frac{1}{N} \Tr\left[(\mathbf{A}-\mathbf{B})(\mathbf{A}-\mathbf{B})^*\right] .
\end{equation}
Here $\A^*$ denotes the conjugate transpose of $\A.$ 
Moreover, if $\A$ and $\mathbf{B}$ are two Hermitian matrices of size $N\times N$, then
\begin{equation}\label{HW2}
\sum_{i=1}^N\left(\lambda_i(\A)-\lambda_i(\mathbf{B})\right)^2 \leq \Tr[(\A-\mathbf{B})^2].
\end{equation}
\end{proposition}

The next two straightforward lemmas control the tail of the product of two Pareto random variables and the expectation of a truncated Pareto.
\begin{lemma}\label{lemma:twotails}
Let $X$ and $Y$ be two independent Pareto r.v.'s with parameters $\beta_1$ and $\beta_2$ respectively, with $\beta_1\leq \beta_2$. There exist constants $c_1=c_1(\beta_1,\beta_2)>0$ and $c_2=c_2(\beta_1)>0$ such that 
$$
\mathbf P(XY>t) 
    = \begin{cases} c_1t^{-\beta_1} & \text{ if }\beta_1< \beta_2\\
c_2t^{-\beta_1}\log t & \text{ if } \beta_1= \beta_2.
\end{cases}
$$
\end{lemma}
\begin{lemma}\label{lem:Par_trunc}
    Let $X$ be a  Pareto random variable with law $\mathbf{P}$ and parameter $\beta>1$. For any $m>0$ it holds
    \[
    \mathbf E\left[X\one_{X\geq m}\right]=\frac{\beta}{(\beta-1)} m^{1-\beta}. 
    \]
\end{lemma}
We state one final auxiliary lemma related to the approximation of sums by integrals.
\begin{lemma}\label{lem:sumtoint}
    Let $\beta\in (0,\,1]$. Then there exists a constant $c_1=c_1(\beta)>0$ such that 
    \begin{equation}\label{eq:alpha_sum_bound}
    \frac{1}{N}\sum_{i\neq j\in \Ver_N}\frac{1}{\|i-j\|^\beta}\sim c_1 \max\{N^{1-\beta},\,\log N\}.
    \end{equation}
    If instead $\beta>1$, there exists a constant $c_2>0$ such that 
    \[
    \frac{1}{N}\sum_{i\neq j\in \Ver_N}\frac{1}{\|i-j\|^\beta}\sim c_2\, .
    \]
\end{lemma}

We end this section by quoting, for the reader's convenience, the following lemma from \citet[ Fact 4.3]{Chakrabarty:Hazra:Sarkar:2016}.
\begin{lemma}\label{lemma:slutsky} Let $(\Sigma, d)$ be a complete metric space, and let $(\Omega, \mathcal{A}, P)$ be a probability space. Suppose that $\left(X_{m n}:(m, n) \in\{1,2, \ldots, \infty\}^2 \backslash\{\infty, \infty\}\right)$ is a family of random elements in $\Sigma$, that is, measurable maps from $\Omega$ to $\Sigma$, the latter being equipped with the Borel $\sigma$-field induced by $d$. Assume that
\begin{enumerate}[label=(\arabic*),ref=(\arabic*)]
\item\label{item:1ar}for all fixed $1 \leq m<\infty$
$$
 \lim_{n\to \infty}d\left(X_{m n}, X_{m \infty}\right) =0 \text{ in $P$-probability}.$$
\item\label{item:2ar} For all $\varepsilon>0$,
$$
\lim _{m \rightarrow \infty} \limsup _{n \rightarrow \infty} P\left(d\left(X_{m n}, X_{\infty n}\right)>\varepsilon\right)=0 .
$$
\end{enumerate}

Then, there exists a random element $X_{\infty \infty}$ of $\Sigma$ such that
\begin{equation}\label{eq:450}
\lim_{m\to\infty}d\left(X_{m \infty}, X_{\infty \infty}\right) = 0 \text{ in $P$-probability}
\end{equation}
and
$$
 \lim_{n\to \infty}d\left(X_{\infty n}, X_{\infty \infty}\right) = 0\text{ in $P$-probability}.
$$
Furthermore, if $X_{m \infty}$ is deterministic for all $m$, then so is $X_{\infty \infty}$, and  \eqref{eq:450} simplifies to
\begin{equation}\label{eq:451}
\lim _{m \rightarrow \infty} d\left(X_{m \infty}, X_{\infty \infty}\right)=0 .
\end{equation}

\end{lemma}

\begin{lemma}[Fact A.4 \cite{CHHS}]\label{lemma:free prod} Suppose that $W_N$ is an $N \times N$ scaled standard Gaussian Wigner matrix, i.e., a symmetric matrix whose upper triangular entries are i.~i.~d. normal with mean zero and variance $1 / N$. Let $D_N^1$ and $D_N^2$ be (possibly random) $N \times N$ symmetric matrices such that there exists a deterministic $C$ satisfying

$$
\sup _{N \geq 1, i=1,2}\left\|D_N^i\right\| \leq C<\infty
$$

where $\|\cdot\|$ denotes the usual matrix norm (which is same as the largest singular value for a symmetric matrix). Furthermore, assume that there is a $W^*$-probability space $(\mathcal{A}, \varphi)$ in which there are self-adjoint elements $d_1$ and $d_2$ such that, for any polynomial $p$ in two variables, it

$$
\lim _{N \rightarrow \infty} \frac{1}{N} \operatorname{Tr}\left(p\left(D_N^1, D_N^2\right)\right)=\varphi\left(p\left(d_1, d_2\right)\right) \text { a.s. }
$$

Finally, suppose that $\left(D_N^1, D_N^2\right)$ is independent of $W_N$. Then there exists a self-adjoint element $s$ in $\mathcal{A}$ (possibly after expansion) that has the standard semicircle distribution and is freely independent of $\left(d_1, d_2\right)$, and is such that

$$
\lim _{N \rightarrow \infty} \frac{1}{N} \operatorname{Tr}\left(p\left(W_N, D_N^1, D_N^2\right)\right)=\varphi\left(p\left(s, d_1, d_2\right)\right) \text { a.s. }
$$

for any polynomial $p$ in three variables.
\end{lemma}

\begin{lemma}[Wick's formula]\label{lemma:Wick}
Let $(X_1, X_2, \ldots, X_n)$ be a real Gaussian vector, then, and $\mathcal{P}_2(k)$ the set of pair partitions of $[k]$. Then, for any $1\le k\le n$, 
\begin{equation}\label{eq:Wick}
\E[X_{i_1}\cdots X_{i_k}]= \sum_{\pi\in \mathcal P_2(k)} \prod_{(r,s)\in \pi} \E[ X_{i_r} X_{i_s}]\,.
\end{equation}
\end{lemma}

\begin{definition}[Graph associated to a partition,~{\citet[Definition 2.3]{avena}}]\label{pap1-partitiongraph}  For a fixed $k\geq 1$, let $\gamma$ denote the cyclic permutation $(1,2,\ldots,k)$. For a partition $\pi$, we define $G_{\gamma\pi}=(\Ver_{\gamma\pi}, E_{\gamma\pi})$ as a rooted, labelled directed graph associated with any partition $\pi$ of $[k]$, constructed as follows.
    \begin{itemize}
        \item Initially consider the vertex set $\Ver_{\gamma\pi}=[k]$ and perform a closed walk on $[k]$ as $1\to 2\to 3\to \cdots \to k\to 1$ and with each step of the walk, add an edge. 
        \item Evaluate $\gamma\pi$, which will be of the form $\gamma\pi = \{V_1,V_2,\ldots,V_m\}$ for some $m\geq 1$ where $\{V_i\}_{1\leq i\leq m}$ are disjoint blocks. Then, collapse vertices in $\Ver_{\gamma\pi}$ to a single vertex if they belong to the same block in $\gamma\pi$, and collapse the corresponding edges. Thus, $\Ver_{\gamma\pi} = \{V_1,\ldots,V_m\}$. 
        \item Finally root and label the graph as follows. 
        \begin{itemize}
            \item \emph{Root:} we always assume that the first element of the closed walk (in this case `1') is in $V_1$, and we fix the block $V_1$ as the root. 
            \item \emph{Label:} each vertex $V_i$ gets labelled with the elements belonging to the corresponding block in $\gamma\pi$.
        \end{itemize}
        \end{itemize} 
\end{definition}

\bibliographystyle{abbrvnat}
\bibliography{reference.bib}
\end{document}